\documentclass[11pt]{amsart}

\usepackage{amsmath,amssymb,amsthm,mathrsfs}
\usepackage{comment}
\usepackage[unicode,breaklinks=true,colorlinks=true,citecolor=OliveGreen]{hyperref}
\usepackage[dvipsnames]{xcolor}
\usepackage{graphicx}

\usepackage[top=1in, bottom=1in, left=1in, right=1in, marginparwidth=1in, marginparsep=0.1in]{geometry}

\numberwithin{equation}{section}
\newtheorem{theorem}{Theorem}[section]
\newtheorem{proposition}[theorem]{Proposition}
\newtheorem{lemma}[theorem]{Lemma}
\newtheorem{definition}[theorem]{Definition}

\theoremstyle{remark}
\newtheorem{remark}[theorem]{Remark}

\definecolor{darkblue}{rgb}{0,0,0.7}
\newcommand{\CPEDIT}{\color{black}}

\newcommand{\abs}[1]{\left| #1 \right|}

\newcommand{\mV}{\mathcal V}
\newcommand{\mF}{\mathcal F}

\newcommand{\xss}{x_{s,s^*}}

\newcommand{\eps}{\varepsilon}

\newcommand*{\defeq}{\mathrel{\vcenter{\baselineskip0.5ex \lineskiplimit0pt
                     \hbox{\scriptsize.}\hbox{\scriptsize.}}}%
     		     =}
\newcommand*{\backdefeq}{=\mathrel{\vcenter{\baselineskip0.5ex \lineskiplimit0pt
                     \hbox{\scriptsize.}\hbox{\scriptsize.}}}%
     		     }

\newcommand{\innerprod}[2]{\left\langle #1,#2 \right\rangle}

\newcommand{\R}{{\mathbb R }}
\newcommand{\N}{{\mathbb N}}

\newcommand{\donothing}[1]{{}}

\makeatletter
\newcommand{\xRightarrow}[2][]{\ext@arrow 0359\Rightarrowfill@{#1}{#2}}
\makeatother

\newcommand{\E}{\mathbb E}
\let\OLDthebibliography\thebibliography
\renewcommand\thebibliography[1]{
  \OLDthebibliography{#1}
  \setlength{\parskip}{1pt}
  \setlength{\itemsep}{1pt plus 0.3ex}
}

\title[Epidemiology with Human Behavior and Stochastic Effects]{A Compartmental Model for Epidemiology with Human Behavior and Stochastic Effects}
\author{Christian Parkinson and Weinan Wang}
\date{\today}

\begin{document}

\begin{abstract}
We propose a compartmental model for epidemiology wherein the population is split into groups with either comply or refuse to comply with protocols designed to slow the spread of a disease. Parallel to the disease spread, we assume that noncompliance with protocols spreads as a social contagion. We begin by deriving the reproductive ratio for a deterministic version of the model, and use this to fully characterize the local stability of disease free equilibrium points. We then append the deterministic model with stochastic effects, specifically assuming that the transmission rate of the disease and the transmission rate of the social contagion are uncertain. We prove global existence and nonnegativity for our stochastic model. Then using suitably constructed stochastic Lyapunov functions, we analyze the behavior of the stochastic system with respect to certain disease free states. We demonstrate all of our results with numerical simulations. 
\end{abstract}

\maketitle

\section{Introduction}

In this manuscript, we design and analyze deterministic and stochastic ordinary differential equation (ODE) epidemiological models incorporating human behavior. Our primary modeling assumptions---drawn from research by social scientists \cite{SS3,SS2,SS1}---are that while governing bodies will enact non-pharmaceutical intervention (NPI) protocols to stunt the spread of a disease, a nontrivial portion of the population will not comply with these protocols, and this noncompliance will have a nontrivial effect on disease spread. We also assume that, in the manner of a social contagion \cite{socialContagion}, the inclination against complying with public health mandates can spread throughout the population. Further, we assume there is uncertainty in transmission, both of the disease and of the social contagion representing noncompliance. 

Our model includes spread of noncompliance via mass action, and is similar to that analyzed in \cite{BPB} in the ODE setting and \cite{BPW, PW} in the partial differential equation (PDE) setting. We discuss the specifics of these models shortly. Pant et al. \cite{Gumel} propose a similar model, where populations are divided into subclasses based on adherence with NPIs, but the transfer between the classes is modeled using linear terms, as opposed to mass action terms. We opt for the latter, though it complicates the analysis, because of the mechanistic precedent from social contagion theory. Other approaches for including behavioral effects in  mathematical epidemiology include modeling of heterogeneous risk aversion (wherein there is behavioral diffusion) \cite{socialDiffusion}, kinetic models including imperfect adherence with NPIs \cite{imperfectAdherence}, behavioral changes based on local disease incidence \cite{incidenceBased}, and age-stratified models whose implicit assumption is the populations of different ages will exhibit significantly different behaviors \cite{ageStratified,pang1,pang3}. {\CPEDIT We are specifically inspired to noncompliance with NPIs as a social contagion due to reports such as \cite{Angus,Nature} wherein roughly 10\% of individuals not adhering to NPIs list their primary reason for not adhering as some variation of "none of my friends are adhering," suggests that peer imitation and social pressure play a significant role in NPI adoption.} For a recent reflection on incorporation of human behavior in epidemiological models, see \cite{GeneralHB}.

Parallel to the work incorporating human behavior, there has been  research on the stochastic compartmental epidemic models. Earlier work by Tornatore et al. \cite{vetro}, Gray et al. \cite{Gray}, and Ji et al. \cite{JJS} appended basic susceptible-infected-recovered (SIR) type models with uncertainty in the disease transmission rate. Later work on stochastic epidemic models includes analysis and/or numerical study of stochastic PDE models under a variety of different modeling decisions \cite{ahmed2023dynamical,li2023stochastic,nguyen2019stochastic}, and study of seasonality of disease emergence in stochastic ODE models \cite{linda}. For overviews on more general stochastic modeling of epidemics, see \cite{linda2,linda3}. Notably, Pang and Pardoux \cite{pang2} formulate a stochastic PDE model which is age-stratified, hence implicitly incorporating some heterogeneous behavior. 

{\CPEDIT This work lies at the intersection of compartmental models for epidemiology incorporating human behavior \`a la \cite{BPB} and those incorporating stochastic effects \`a la \cite{Gray,JJS,vetro}. To the best of the authors' knowledge, this manuscript is the first to model behavioral effects in such a mechanistic manner, while also considering stochastic perturbation. As such, this is an extension of the work in each of \cite{BPB,Gray,JJS,vetro}.}

The remainder of this manuscript is organized as follows. In section \ref{sec:det}, we present a deterministic ODE model for epidemiology which includes NPIs and noncompliant behavior, and give a full analysis of the stability of disease free equilibrium points for the deterministic model. In section \ref{sec:stoch1}, we append the model with stochastic perturbations, and prove that solutions of our new system exist globally and remain positive as long as there is positive initial data. In section \ref{sec:stoch2}, we provide stability analysis for the disease free equilibria of the stochastic system, and draw analogies between the deterministic and stochastic settings. Our main strategy for establishing asymptotic behavior of the stochastic system will be the construction of suitable stochastic Lyapunov functions in different parameter regimes. In section \ref{sec:sim}, we demonstrate all of our results with simulations, and we finish with brief concluding remarks and comments regarding avenues of future work in section \ref{sec:conclusion}.

As a final note before commencing, we mention that for our stochastic model in sections \ref{sec:stoch1} and \ref{sec:stoch2}, we opt to perturb the deterministic model using Gaussian white noise with the It\^o interpretation as in \cite{Gray,JJS,vetro}. This is to distinguish from the Stratonovich calculus approach used, for example, in \cite{BL,CMZ,FWZ,LM}. Braumann \cite{Braumann} elucidates the difference between these approaches, and concludes that, in essence, it is ``merely semantic," stemming from the informal interpretation of phrases like \emph{average per capita growth rate} or \emph{average infection rate}. In the context of SIR type models, the difference usually presents itself mathematically in the following way. When using It\^o noise, the strength of perturbation typically appears directly in the conditions for disease extinction or persistence as seen below and in \cite{Gray,JJS,vetro}. By contrast, when using Stratonovich noise, the randomness does not cause a change in the reproductive ratio or the extinction criteria, but rather changes the bounds on the infected population in the case of disease persistence; see, for example \cite{FWZ,LM}. The It\^o and Stratonovich interpretations have some theoretical tradeoffs as well. For example, in the It\^o formulation for any reasonably well-behaved function $f$, the integral $\int^t_0 f dW_t$ is a continuous local martingale, while the same is not true in the Stratonovich formulation. However, a stochastic ODE with Stratonovich noise can be easily realized as a limit of a specific type of perturbed deterministic ODE via the Wong-Zakai theorem \cite{WZ,LM}, while the same is not true when using It\^o noise. In another direction, some authors introduce noise via an Ornstein-Uhlenbeck process \cite{LP,LiuOU} which has the mean-reverting property making the possibility of negative perturbation less likely \cite{MF2,MF}, as indeed d'Onofrio \cite{dnoisy} points out that in some biological models, Gaussian white noise can lead to questionable results. However, we choose white noise to reflect that random effects can potentially raise or lower infectivity rates, and neither our theorems nor our simulations indicate biologically absurd results. Instead, our results imply that while the random perturbations can raise or lower infectivity, to ensure disease extinction, the policy-maker must ``plan for the worst", though individual paths for the stochastic ODE may reflect better or worse outcomes than in the corresponding deterministic case. All this to say that while there are many manners of introducing stochastic perturbation, following \cite{Gray,JJS,vetro}, we use simple Gaussian white noise and leave the integration of other forms of randomness into our model as an avenue of future work.

\section{A deterministic SIR model with noncompliant behavior} \label{sec:det}

To build our model, we begin with the basic SIR model of Kermack-McKendrick \cite{OG} assuming a natural birth rate of $b>0$, a natural death rate of $\delta > 0$, a transmission rate of $\beta > 0$ and a recovery rate of $\gamma>0$: \begin{equation}\label{eq:basicSIR}
\begin{split}
	\frac{dS}{dt} &= b-\beta SI - \delta S,\\
	\frac{dI}{dt} &= \beta SI - (\gamma+\delta) I, \\
    \frac{dR}{dt} &= \gamma I - \delta R. 
	\end{split}
 \end{equation} We modify \eqref{eq:basicSIR} in several ways, and in doing so, introduce several new variables and parameters. Thus, in addition to the discussion in the ensuing paragraph, to aid the reader, we include descriptions of all variables and parameters in figure \ref{fig:params} and include a flow diagram for our system in figure \ref{fig:flowDiagram}.
 
 {\CPEDIT Following the modeling strategy of \cite{BPB,BPW,PW,P1}, we imagine a scenario where the government enacts non-pharmaceutical intervention (NPI) measures, such as mask-wearing, social distancing, or shelter-at-home orders, which decrease the infection rate by decreasing the amount of mixing by some portion $\alpha \in [0,1]$ for those who comply. We then split each compartment $S,I,R$ into compliant individuals (who retain the labels $S,I,R$) and noncompliant individuals (henceforth labeled $S^*,I^*,R^*$). The noncompliant individuals do not receive a reduction in mixing. Borrowing from social contagion theory, we imagine that noncompliance also spreads via mass action with rate $\mu>0$, parallel to the disease spread. While some research suggests the possibility that NPI fatigue may make it unlikely for noncompliant individuals to readopt compliant behavior \cite{RefForNu3}, other work suggests that well-implemented public advocacy programs can increase compliance with NPIs \cite{RefForNu1, RefForNu2}. Accordingly, for full generality, we allow noncompliant individuals to become compliant again with transfer rate $\nu\ge 0.$ We note that this effect can be removed by setting $\nu = 0$ if one prefers and all subsequent derivations and theorems still hold.}
 
 Finally, we assume that any newly added members of the population are noncompliant with probability $\xi \in [0,1]$. Defining the \emph{actively mixing} infectious population by $I^{(M)} = (1-\alpha)I+I^*$ and the total noncompliant population $N^* = S^* + I^* + R^*$. All of these modeling decisions lead to the equations \begin{equation} \label{eq:SIRwithCompliance} \begin{split}
	\frac{dS}{dt} &= (1-\xi) b - \beta(1-\alpha)SI^{(M)} -\mu SN^* + \nu S^* - \delta S,\\
	\frac{dI}{dt} &= \beta (1-\alpha)SI^{(M)} - \gamma I - \mu IN^* + \nu I^* - \delta I, \\
	\frac{dR}{dt} &= \gamma I - \mu RN^* + \nu R^*- \delta R,\\
	\frac{dS^*}{dt}&= \xi b -\beta S^*I^{(M)} + \mu SN^* - \nu S^* - \delta S^*,\\
	\frac{dI^*}{dt} &= \beta S^*I^{(M)} - \gamma I^* +\mu IN^* - \nu I^* - \delta I^*, \\
	\frac{dR^*}{dt} &= \gamma I^* + \mu RN^* - \nu R^* - \delta R^*.
	\end{split}
	\end{equation} 

    \begin{figure}[b!]
\centering
\begin{tabular}{|l l|}
\hline
    Quantity & Description\\
    \hline
    $S$ & compliant susceptible population\\
    $I$ & compliant infectious population\\
    $R$ & compliant recovered population\\
    $S^*$ & noncompliant susceptible population\\
    $I^*$ & noncompliant infectious population\\
    $R^*$ & noncompliant recovered population\\
    $N^*$ & total noncompliant population ($N^* = S^* + I^* + R^*)$\\ 
    $b$ & natural birth rate $(b > 0)$\\
    $\xi$ & portion of newly introduced population which is noncompliant ($\xi \in [0,1]$)\\
    $\delta$ & natural death rate ($\delta > 0$)\\
    $\beta$ & natural infectivity strength of the disease $(\beta > 0)$\\
    $\gamma$ & natural recovery rate for the disease $(\gamma > 0)$\\
    $\alpha$ &  reduction in infectivity due to compliance with NPIs ($\alpha \in (0,1]$)\\
    $\mu$ & baseline infectivity strength of noncompliance ($\mu > 0$)\\ 
    $\nu$ & maximal achievable rate of recovery from noncompliance ($\nu \ge 0)$\\
    \hline
\end{tabular}
\caption{A list of variables and parameters for \eqref{eq:SIRwithCompliance}.}
\label{fig:params}
\end{figure}

System \eqref{eq:SIRwithCompliance} is the deterministic version of the model that we are interested in. It is very similar to the compartmental model proposed in \cite{BPB} (ours is simpler in that there is no asymptomatic infected class, but more complicated due to the inclusion of recovery from noncompliance, as well as birth and death terms). Before we introduce the stochastic model (\eqref{eq:StochSIR} below), we prove some results for the deterministic model \eqref{eq:SIRwithCompliance} which we compare to the stochastic model later. This analysis largely follows \cite{BPB}, though we give a fuller picture of the disease free equilibria and their stability properties. 

First notice that, defining $N_{\text{total}} = S+I+R+S^*+I^*+R^*$, we have \begin{equation} \label{eq:totalPop}\dot N_{\text{total}} = b - \delta N_{\text{total}}.\end{equation} Without loss of generality, we can normalize so that $N_{\text{total}}(0)=1$ and thus \begin{equation} \label{eq:totalPopBound} N_{\text{total}}(t) = \frac b \delta + \left(1-\frac b \delta \right)e^{-\delta t}.\end{equation} From here, we will assume that $\tfrac b \delta \ge 1$, so that the total population remains bounded by $b/\delta$ and tends to $b/\delta$ as $t\to \infty$.
\begin{figure}[t!]
    \includegraphics[width=0.7\textwidth]{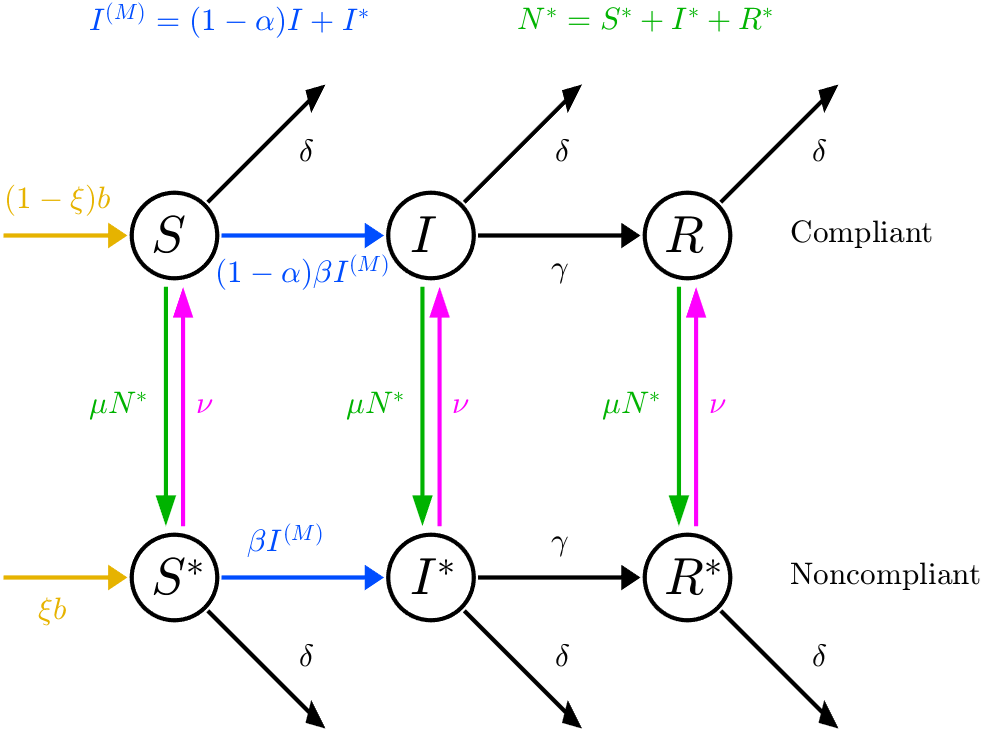}
    \caption{The flow diagram for \eqref{eq:SIRwithCompliance}. Arrows leaving a given population denote outward flow proportional to that population with the given rate. Colored arrows connote population transfer which is modified or introduced for our model (i.e., manners in which our model substantively differs from the basic SIR model).}
    \label{fig:flowDiagram}
\end{figure}

We introduce some notation which is standard when deriving basic reproductive ratios for compartmental epidemiology models using the next generation matrix method \cite{NextGenMat1,NextGenMat2,NextGenMat3}. We represent the solution of \eqref{eq:SIRwithCompliance} by \begin{equation} \label{eq:ordering}x = (x_1,x_2,x_3,x_4,x_5,x_6)= (I,I^*,S,R,S^*,R^*). \end{equation} Note that when writing the solution this way, we have moved the infectious compartments to the front, keeping others in the order in which they appear in \eqref{eq:SIRwithCompliance}. Then the system can be written \begin{equation}
    \label{eq:SIRNGMformat} \dot x = \mathcal F(x) - \mathcal V(x),
\end{equation} where $\mF,\mV:\R^6 \to \R^6$. Here $\mathcal F_i$ contains all terms which introduce new infections into compartment $i$, and $\mathcal V_i$ contains all other transfer into and out of compartment $i$. We then further decompose $\mathcal V(x) = \mV^-(x) - \mV^+(x)$, where $\mV^-_i$ corresponds to flow out of compartment $i$ and $\mV^+_i$ corresponds to any flow into compartment $i$ which does not introduce new infections. For our particular model, the infection function is 
\begin{equation} \label{eq:mF}
    \mF(x) = \begin{bmatrix} \beta(1-\alpha)S((1-\alpha)I+I^*) \\ \beta S^*((1-\alpha)I+I^*) \\ 0 \\ 0 \\ 0 \\ 0\end{bmatrix},
\end{equation} 
and the transfer functions are 
\begin{equation} \label{eq:mV}
    \mV^+(x) = \begin{bmatrix} \nu I^* \\ \mu IN^* \\ (1-\xi) b + \nu S^* \\ \gamma I + \nu R^* \\ \xi b + \mu SN^* \\ \gamma I^* + \mu RN^* \end{bmatrix}, \,\,\,\,\, \mV^-(x) = \begin{bmatrix} \gamma I + \mu IN^* + \delta I \\ \gamma I^* + \nu I^* + \delta I^* \\ \beta(1-\alpha)S((1-\alpha)I+I^*) + \mu SN^* + \delta S \\ \mu RN^* + \delta R \\ \beta S^*((1-\alpha)I+I^*)+\nu S^* + \delta S^* \\ \nu R^* + \delta R^*\end{bmatrix}.
\end{equation} With nonnegative initial conditions, the solutions of \eqref{eq:SIRwithCompliance} will remain nonnegative, so our state space is $\overline\R^6_+$. We define the disease free (DF) set $$U_{DF} = \{x \in \overline\R^6_+ \, : x_1 = x_2 = 0\}$$ and the set of admissible disease free states by \begin{equation} \label{eq:ICset}U_0 = \{x \in \overline\R^6_+ \, : \, x_3 + x_5 \le \tfrac b \delta, x_1 = x_2 = x_4 = x_6 = 0\} \subset U_{DF}.\end{equation}  We are interested in the stability of $U_0$ with regards to $U_{DF}$. Notice that points $x_0 \in U_0$, even if their components sum to $b/\delta$, are \emph{not} necessarily equilibrium points for \eqref{eq:SIRwithCompliance}, since there is still transfer between $S$ and $S^*$. Because of this, we need to adapt our definition of stability. Following \cite{stableSetDFE}, we define stability for $T \subset U_{DF}$ in the following sense. \\

\begin{definition}[Stability of $U \subset U_{DF}$] 
We say that $U \subset U_{DF}$ is \emph{disease-free-stable} (DF-stable) if there is a neighborhood $Z$ of $U$ in $\overline\R^6_+$ such that any solution $x(t)$ of \eqref{eq:SIRNGMformat} with $x(0) \in Z$ satisfies $$\abs{x_1(t)}, \abs{x_2(t)} \le c_1 \max\{\abs{x_1(0)},\abs{x_2(0)}\} e^{-c_2t}$$ for some constants $c_1,c_2>0$ which are independent of $x(0).$ We say that a point $x_0 \in U_{DF}$ is DF-stable if the singleton set $\{x_0\}$ is DF-stable.
\end{definition} 

Intuitively, $U \subset U_{DF}$ is disease-free-stable if dynamics beginning  near enough to $U$ tend toward $U_{DF}$ exponentially fast as $t\to\infty$. In particular, for an equilibrium point $x_0 \in U_{DF}$, the classical notion of local asymptotic stability implies this notion disease free stability for $\{x_0\}$. 

{\CPEDIT To apply the definition as in \cite{stableSetDFE}, we linearize \eqref{eq:SIRwithCompliance} around a point $x_{s,s^*} = (0,0,s ,0,s^*,0) \in X_0$ satisfying $s+s^* \le b/\delta$. We see}
\begin{equation}
D\mF(\xss) = \begin{bmatrix} \beta(1 -\alpha)^2 s & \beta(1-\alpha)s & \cdots \\ \beta (1-\alpha)s^* & \beta s^* & \\ \vdots & & \ddots \end{bmatrix}
,
\end{equation}
where all other entries of are $D\mF(\xss)$ are zero. Next, \begin{equation} \label{eq:Vp} D\mV^+(\xss) = \begin{bmatrix} 
0 & \nu & 0 & 0 & 0 & 0 \\
\mu s^* & 0 & 0 & 0 & 0 & 0 \\
0 & 0 & 0 & 0 & \nu & 0 \\
\gamma & 0 & 0 & 0 & 0 & \nu \\
0 & \mu s & \mu s^*& 0 & \mu s & \mu s \\
0 & \gamma & 0 & \mu s^* & 0 & 0
\end{bmatrix} \end{equation} and \begin{equation} \label{eq:Vm}D\mV^-(\xss) = \begin{bmatrix}
    \gamma + \delta + \mu s^* & 0 & 0 & 0 & 0 & 0 \\
    0&\gamma + \delta + \nu &  0 & 0 & 0 & 0 \\
    \beta (1-\alpha)^2s & (\beta(1-\alpha)+\mu)s& \delta + \mu s^* & 0 & \mu s & \mu s \\
    0 & 0 & 0 & \delta + \mu s^* & 0 & 0\\
    \beta (1-\alpha)s^* & \beta s^* & 0 & 0 & \delta + \nu & 0 \\
    0 & 0 & 0 & 0 & 0 & \delta + \nu
\end{bmatrix}.\end{equation} We let $$F_{s,s^*} = \begin{bmatrix} \beta(1-\alpha)^2 s& \beta(1-\alpha)s \\ \beta (1-\alpha)s^* & \beta s^*  \end{bmatrix}, \,\,\,\,\, V_{s,s^*} = \begin{bmatrix} \gamma + \delta + \mu s^*& -\nu \\ -\mu s^* & \gamma + \delta + \nu \end{bmatrix}.$$ 
These are the $2\times 2$ principal minors for $D\mF(\xss)$ and $D\mV(\xss) = D\mV^-(\xss) - D\mV^+(\xss)$, respectively. Then the linearization of the infectious compartments of \eqref{eq:SIRNGMformat} around $\xss$ is \begin{equation} \label{eq:lin} \begin{bmatrix} x_1 \\ x_2 \end{bmatrix}' = (F_{s,s^*}-V_{s,s^*}) \begin{bmatrix} x_1 \\ x_2 \end{bmatrix}. \end{equation} 
We note that $V_{s,s^*}$ is invertible: its determinant $(\gamma + \delta)(\gamma + \delta + \nu + \mu s^*)$ is positive. It is proven in \cite{NextGenMat3} that if $\rho(F_{s,s^*} V_{s,s^*}^{-1}) < 1$, then all eigenvalues of $F_{s,s^*}-V_{s,s^*}$ have negative real part, and thus the linearization \eqref{eq:lin} exhibits exponential decay as $t\to\infty$. This motivates the definition of the reproductive ratio $\mathscr R_0(s,s^*)$ corresponding to the disease free state $\xss$: \begin{equation} \label{eq:R000}
\mathscr{R}_0(s,s^*) = \rho(F_{s,s^*} V_{s,s^*}^{-1}),
\end{equation} where $\rho(\cdot)$ denotes the spectral radius. In our case, $F_{s,s^*}$ is rank 1, so it is easy to explicitly find the spectral radius of $F_{s,s^*} V_{s,s^*} ^{-1}$. After some simplification, we have \begin{equation} \label{eq:R0deterministic}
    \mathscr R_0(s,s^*) = \frac{\beta}{\gamma + \delta} \left(s \left[(1-\alpha)^2 + \frac{\alpha(1-\alpha) \mu s^*}{\gamma + \delta + \nu + \mu s^*}\right] + s^* \left[1 - \frac{\alpha \nu}{\gamma + \delta + \nu + \mu s^*}\right]\right).
\end{equation} {\CPEDIT This formula has a fairly straightforward interpretation as a weighted average of the reproductive ratios for the compliant and noncompliant classes, complicated by the fact that there is some transfer between the two.  Notice in particular that} \begin{equation} \CPEDIT \label{eq:bestWorst} \mathscr R_0(\tfrac b \delta,0) = \frac b \delta \frac{\beta(1-\alpha)^2}{\gamma+\delta}, \,\,\,\,\,\,\,\,\,\,\,\,\,\,\, \mathscr R_0(0,\tfrac b \delta ) = \frac b \delta \cdot \frac{\beta}{\gamma +\delta}\left(1 - \frac{\alpha \nu}{\gamma + \delta + \nu + \mu s^*}\right). \end{equation}  {\CPEDIT The first is the reproductive ratio from the basic SIR model with infection rate $\beta (1-\alpha)^2$ which occurs when the entire population is compliant. The second encapsulates the ``worst case scenario" where the population is at its steady state and the entire population is noncompliant. In this case we lose the $(1-\alpha)^2$ reduction in infectivity, but the reproductive ratio is still slightly decreased due to the possibility of individuals becoming compliant. }

{\CPEDIT Given this definition of $\mathscr R_0(s, s^*)$, we can establish stability properties for disease free equilibria.}

\begin{lemma} \label{lem:R0increase}
As defined in \eqref{eq:R0deterministic}, the maximum value of $\mathscr R_0(s,s^*)$ for $s,s^*\ge 0$ and $s+s^* \le \tfrac b \delta$ occurs at $(s,s^*) = (0,\tfrac b \delta).$
\end{lemma}
\begin{proof}
    Fixing $s^*$, it is clear that $\mathscr R_0(s,s^*)$ is increasing in $s$, since it is linear with a positive slope. A quick calculation shows that $$\frac{d\mathscr R_0}{ds^*} = \frac{\beta}{\gamma + \delta}\left(  \frac{\gamma + \delta + (1-\alpha)\nu + \mu s^*}{\gamma + \delta + \nu + \mu s^*}+ \frac{\alpha (1-\alpha)\mu s(\gamma + \delta + \mu) + \alpha \nu \mu s^*}{(\gamma + \delta + \nu + \mu s^*)^2}\right) > 0$$ which implies that $\mathscr R_0(s,s^*)$ is increasing in $s^*$ for fixed $s$. Thus the maximum must occur somewhere on the boundary line $s+s^* = \tfrac b \delta$. Considering $(s,s^*) = (\tfrac b \delta - \theta, \theta)$ for $\theta \in [0,\tfrac b \delta]$, we see \begin{equation} \label{eq:drdtheta}\begin{split}
        \frac{d}{d\theta}\Big(\mathscr R_0 (\tfrac b \delta - \theta,\theta)\Big) &= \frac{\alpha \beta }{\gamma + \delta }\Big((1-\alpha) + (1-A)\Big)
    \end{split} \end{equation} where $$A = \frac{ (1-\alpha)\mu\theta(\gamma + \delta + \nu + \mu\theta) + \nu(\gamma+\delta + \nu)-\left( \tfrac b \delta - \theta\right)(1-\alpha)\mu(\gamma + \delta + \nu)}{(\gamma + \delta + \nu + \mu\theta)^2}.$$ By dropping the negative term in the numerator, and adding $\mu \theta$ into the parentheses in the second term, we see $A < 1$ and thus \eqref{eq:drdtheta} shows that $\mathscr R_0(\tfrac b \delta-\theta,\theta)$ increases in $\theta$ so that the maximum occurs at $(s,s^*) = (0,\tfrac b \delta)$.
\end{proof}

Given this, we can achieve DF-stability for our admissible set of disease free states defined in \eqref{eq:ICset} with a very strong assumption on $\mathscr R_0(s,s^*)$.

\begin{theorem}
    If $\mathscr R_0(0,\tfrac b \delta) < 1$, then $U_0$ is DF-stable.
\end{theorem}

\begin{proof}
    This follows from lemma \ref{lem:R0increase} and \cite[Theorem 3.1]{stableSetDFE} after noting that the family of linearizations \eqref{eq:lin} satisfy $\rho(F_{s,s^*} V_{s,s^*}^{-1}) = \mathscr R_0(s,s^*)\le \mathscr R_0(0,\tfrac b \delta) < 1$, meaning that \eqref{eq:lin} exhibits exponential decay with decay rate which is uniform over $x_0 \in U_0$. 
\end{proof}

Beyond this, we derive the exact values for the disease free steady states. Inserting $I=I^*=R=R^*=0$ into \eqref{eq:SIRwithCompliance}, we see the disease free steady states $(s,s^*)$ satisfy \begin{equation}
\label{eq:deterministicDFE1}
\begin{split}
    (1-\xi) b-\mu ss^* + \nu s^* -\delta s &= 0,\\
    \xi b + \mu ss^* - \nu s^* -\delta s^* &=0.
\end{split}
\end{equation} Adding the equations in \eqref{eq:deterministicDFE1}, we see that $s+s^*= \frac b \delta$ (which also follows from \eqref{eq:totalPop}). Inserting $x^* = \frac b \delta -s$ into the first equation and rearranging yields the following quadratic equation for the steady states of the compliant population: $$s^2 - \left(\frac b \delta + \frac {\nu+\delta} \mu\right)s + \frac b \delta \left(\frac{\nu+(1-\xi)\delta}{\mu} \right)=0.$$ After some algebra, the roots of this equation can be written \begin{equation} \label{eq:DFEsDeterministic} s_\pm = \frac 1 2\left( \frac{b}{\delta} + \frac{\nu+\delta}{\mu} \pm \sqrt{\left(\frac{b}{\delta} - \frac{\nu+\delta}{\mu} \right)^2 + \frac{4\xi b}{\mu} } \right). \end{equation} The corresponding noncompliant populations at the DFEs are recovered from $s^*_{\pm} = \frac b \delta - s_{\pm}$. We include some observations regarding these (omitting the proofs because they are exercises in simple if tedious algebra). \\

\begin{proposition} \label{prop:DFEs}
If $\xi = 0$, the disease free equilibria $\left(0,0,s,0,s^*,0\right)$ of \eqref{eq:SIRwithCompliance} are given by \begin{equation} \label{eq:DFE12}
x_1 = \left(0,0,\frac b \delta,0,0,0\right), \,\,\,\,\,\,\, 
x_2 = \left(0,0,\frac{\delta + \nu}{\mu},0,\frac{b}{\delta}-\frac{\delta + \nu}{\mu},0\right).
\end{equation} 
The second is only distinct and physically meaningful if $\frac{b}\delta > \frac{\delta + \nu}{\mu}$. The DFEs $x_1$ and $x_2$ correspond (respectively) to the values of $s_+$ and $s_-$ from \eqref{eq:DFEsDeterministic}.

If $\xi \in (0,1]$, then from \eqref{eq:DFEsDeterministic}, we see that $s_+ > \frac b \delta$, which is not physically meaningful. In this case, there is one DFE given by \begin{equation} \label{eq:DFE3}
\begin{split}
x_3 &=  \left(0,0,\tfrac 1 2\left( \tfrac{b}{\delta} + \tfrac{\delta + \nu}{\mu} - \sqrt{\left(\tfrac{b}{\delta} - \tfrac{\delta + \nu}{\mu} \right)^2 + \tfrac{4\xi b}{\mu} } \right), 0,\tfrac 1 2\left( \tfrac{b}{\delta} - \tfrac{\delta + \nu}{\mu} + \sqrt{\left(\tfrac{b}{\delta} - \tfrac{\delta + \nu}{\mu} \right)^2 + \tfrac{4\xi b}{\mu} } \right),0 \right)
\end{split}
\end{equation} 
which is physically meaningful for any positive choices of the parameters. Further, in the case that $\tfrac b \delta \le \frac{\delta + \nu}{\mu}$, $x_3$ will coincide with $x_1$ in the limit as $\xi \to 0$, and in the case that $\tfrac b \delta > \frac{\delta + \nu}{\mu}$, $x_3$ will coincide with $x_2$ in the limit as $\xi \to 0.$
\end{proposition}

With one further condition on $D\mV = D\mV^- - D\mV^+$ (as defined in \eqref{eq:Vp},\eqref{eq:Vm}), we can characterize the local stability of these DFE.

\begin{theorem} \label{eq:stabilityOfDFE} We consider $x_1,x_2,x_3$ as defined in \eqref{eq:DFE12}, \eqref{eq:DFE3}. 
\begin{itemize}
\item[(i)] If $\xi = 0$, and $\frac{b}{\delta} < \frac{\delta + \nu}{\mu}$, then $x_1$ is locally asymptotically stable if $\mathscr R_0(\tfrac b \delta,0) < 1$, and unstable if $\mathscr R_0(\frac b \delta, 0) > 1$. 

\item[(ii)] If $\xi = 0$, and $\frac{b}{\delta} > \frac{\delta + \nu}{\mu}$, then $x_2$ is locally asymptotically stable if $\mathscr R_0(\tfrac{\delta + \nu}{\mu},\tfrac b \delta - \tfrac{\delta + \nu}{\mu}) < 1$, and unstable if $\mathscr R_0(\tfrac{\delta + \nu}{\mu},\tfrac b \delta - \tfrac{\delta + \nu}{\mu}) > 1$. 

\item[(iii)]If $\xi \in (0,1]$, we write $x_3 = (0,0,s_3,0,s^*_3,0)$. In this case, $x_3$ is locally asymptotically stable if $\mathscr R_0(s_3,s_3^*) < 1$, and unstable if $\mathscr R_0(s_3,s^*_3) > 1$, where $x_3 = (0,0,s_3,0,s^*_3,0).$
\end{itemize}
\end{theorem}

\begin{proof}
    Our theorem follows from \cite[Theorem 2]{NextGenMat3}. We use the definitions of $\mathcal F, \mathcal V^{\pm}$ in \eqref{eq:mF},\eqref{eq:mV}, and let $\mathcal F_i$ denote the $i^{\text{th}}$ coordinate of $\mathcal F$ and similarly for $\mathcal V^\pm$. With this notation, assumptions (A1)-(A5) from \cite[Theorem 2]{NextGenMat3} can be written: \begin{itemize} 
    \item[(A1)] $\mathcal F, \mathcal V^{\pm}$ are component-wise nonnegative when their arguments are nonnegative.
    \item[(A2)] For each $i = 1,\ldots, 6$, $\mathcal V^-_i = 0$ when $x_i = 0$. (Recall, that the variables are ordered $x = (x_1,x_2,x_3,x_4,x_5,x_6) = (I,I^*,S,R,S^*,R^*)$.)
    \item[(A3)] For $i=3,4,5,6$, $\mathcal F_i \equiv 0$.
    \item[(A4)] If $x \in X_{DF},$ then $\mathcal F(x) = 0$ and $\mathcal V^+_i(x) =0$ for $i = 1,2.$
    \item[(A5)] At the disease free equilbrium of interest, all eigenvalues of $D\mathcal V = D \mathcal V^- - D\mathcal V^+$ have positive real part. 
    \end{itemize} Note that conditions (A1)-(A4) must hold generally. The last condition (A5) is local to the particular DFE one is analyzing. Under these conditions, \cite[Theorem 2]{NextGenMat3}, gives local asymptotic stability of the DFE when $\mathscr R_0 <1$, and instability when $\mathscr R_0 > 1.$ 

    It is easily seen that (A1)-(A4) hold for our system. The only assumption that could possibly fail is (A5). From \eqref{eq:Vp},\eqref{eq:Vm}, at a disease free point $x_{s,s^*}$ with $s,s^*\ge 0$ and $s+s^* \le \tfrac b \delta$, we have \begin{equation}
        \label{eq:Jacobian} 
        D\mathcal V(x_{s,s^*}) = \begin{bmatrix}
    \gamma + \delta + \mu s^* & -\nu & 0 & 0 & 0 & 0 \\
    -\mu s^* &\gamma + \delta + \nu &  0 & 0 & 0 & 0 \\
    \beta (1-\alpha)^2s & (\beta(1-\alpha)+\mu)s& \delta + \mu s^* & 0 & \mu s-\nu & \mu s \\
    -\gamma & 0 & 0 & \delta + \mu s^* & 0 & -\nu\\
    \beta (1-\alpha)s^* & \beta s^*-\mu s & -\mu s^* & 0 & \delta + \nu-\mu s & -\mu s \\
    0 & -\gamma  & 0 & -\mu s^* & 0 & \delta + \nu
\end{bmatrix}.
    \end{equation} By zooming on the first principle $2\times 2$ submatrix in \eqref{eq:Jacobian}, we observe that there are two eigenvalues of $[D\mathcal V(x_{s,s^*})]^T$ which have eigenvectors of the form $(z_1,z_2,0,0,0,0)$. These must also be eigenvalues of $D\mathcal V(x_{s,s^*})$ and can be easily found to be \begin{equation}\label{eq:eig12}\lambda_1 = \gamma + \delta, \hspace{1cm} \lambda_2 = \gamma + \delta + \nu + \mu s^*.\end{equation} The other eigenvalues of $D\mathcal V(x_{s,s^*})$ have eigenvectors of the form $(0,0,z_3,z_4,z_5,z_6)$, and can be found by zooming into the last principal $4\times 4$ submatrix. They are given by \begin{equation} \label{eq:eig36}
    \lambda_3 = \delta, \,\,\,\,\,\,\,\,\,\,  \lambda_4 = \delta, \,\,\,\,\,\,\,\,\,\,  \lambda_5 =  \delta + \nu + \mu s^*, \,\,\,\,\,\,\,\,\,\,  \lambda_6 = \delta + \nu + \mu (s^*-s).
    \end{equation} Eigenvalues $\lambda_1,\ldots,\lambda_5$ are positive, so (A5) is reduced to the condition that $\lambda_6 > 0$ which is equivalent to \begin{equation} \label{eq:stabCond}
    s - s^* < \frac{\delta + \nu}{\mu}.
    \end{equation} We consider the DFE values in \eqref{eq:DFE12},\eqref{eq:DFE3}. 
    
    For $\xi = 0$ and $x_1 = (0,0,b/\delta,0,0,0)$, \eqref{eq:stabCond} reduces to $\frac{b}{\delta} < \frac{\delta + \nu}{\mu}$. Under this condition we have local asymptotic stability of $x_1$ when $\mathscr R_0 (\tfrac b \delta,0) <1,$ and instability when $\mathscr R_0(\tfrac b \delta,0)>1$.

    For $\xi = 0$ and $x_2 = (0,0,\tfrac{\delta + \nu}{\mu},0,\tfrac b \delta - \tfrac{ \delta + \nu}\mu,0)$, \eqref{eq:stabCond} reduces to $$\left(\frac{\delta + \nu}{\mu} - \left(\frac b \delta - \frac{\delta + \nu}{\mu}\right)\right) < \frac{\delta + \nu}{\mu} \,\,\,\,\, \Longleftrightarrow \,\,\,\,\, \frac b \delta > \frac {\delta + \nu}\mu.$$ Under this condition, we have local asymptotic stability of $x_2$ when $\mathscr R_0(\tfrac{\delta + \nu}{\mu},\tfrac b \delta - \tfrac{\delta + \nu}{\mu}) < 1$, and instability when $\mathscr R_0(\tfrac{\delta + \nu}{\mu},\tfrac b \delta - \tfrac{\delta + \nu}{\mu}) > 1$.

    For $\xi \in (0,1]$ and $x_3$ as defined in \eqref{eq:DFE3}, \eqref{eq:stabCond} reduces to $$\frac{\delta + \nu}{\mu} - \sqrt{\left(\frac{b}{\delta} - \frac{\delta + \nu}{\mu} \right)^2 + \frac{4\xi b}{\mu}} < \frac{\delta + \nu}{\mu},$$ which holds for any positive values of the parameters, so writing $x_3 = (0,0,s_3,0,s_3^*,0),$ where $s_3,s^*_3$ are as defined in \eqref{eq:DFE3}, $x_3$ is locally asymptotically stable when $\mathscr R_0(s_3,s^*_3)<1$, and unstable when $\mathscr R_{0}(s_3,s_3^*)>1$. \end{proof}

    \begin{remark} \label{rem:Thm25}
    In particular, theorem \ref{eq:stabilityOfDFE}(iii) addresses the special case of $\xi = 1$ and $\nu = 0$, which could be thought of as the worst case scenario, where all newly introduced members of the population are noncompliant and noncompliance is a permanent state. In this case, $x_3 = (0,0,0,0,\tfrac b \delta,0)$ so that the entire population is noncompliant at the DFE, and $\mathscr R_0(0,\tfrac b \delta) = \tfrac b \delta \cdot \frac \beta{\gamma + \delta}$ which is the same reproductive ratio as in \eqref{eq:basicSIR}. This is the case where governmental protocols have no bearing because in the long run, no one will comply with them. 
    \end{remark}

\section{An SIR model with noncompliant behavior and stochastic perturbation} \label{sec:stoch1}

We now present a version of \eqref{eq:SIRwithCompliance} wherein there is uncertainty in the infection rates $\beta$ and $\mu$ for the disease and noncompliance, respectively. In what follows, we let $W = W(t)$ be a a scalar Brownian motion defined on a complete probability space $(\Omega,\mathscr F, \{\mathscr F_t\}_{t\ge 0}, P\}$ where the filtration $\{\mathscr F_t\}_{t\ge 0}$ is increasing and right continuous and $\mathscr F_0$ contains all $P$-null sets.

In essence, we would like to replace $\beta$ with $\beta + \sigma_\beta \dot W$ where $\sigma_\beta \ge 0$ is the level of uncertainty in $\beta$. Similarly, we would like to replace $\mu$ with $\mu +\sigma_\mu \dot W$ for some level of uncertainty $\sigma_\mu\ge 0$. However, for technical reasons involving the proof of nonnegativity and global existence of solutions, it is desirable to ensure that the uncertainty in any given equation is proportional to the population which the equation describes. Thus we consider the following dynamics: \begin{equation} \label{eq:StochSIR}
\begin{split}
	dS &= ((1-\xi) b - \beta(1-\alpha)SI^{(M)} -\mu SN^* + \nu S^* - \delta S)dt+(-\sigma_\beta(1-\alpha)^2SI -\sigma_\mu SS^*)dW,\\
	dI &= (\beta(1-\alpha)SI^{(M)} - \gamma I -\mu I N^* + \nu I^* - \delta I)dt +(\sigma_\beta(1-\alpha)^2SI -\sigma_\mu II^*)dW, \\
    dR &= (\gamma I - \mu RN^* + \nu R^* - \delta R)dt - \sigma_\mu RR^*dW,\\
    dS^* &= (\xi b - \beta S^* I^{(M)} + \mu SN^* - \nu S^* -\delta S^*)dt +(-\sigma_\beta S^* I^* + \sigma_\mu SS^*)dW,\\
	dI^* &= (\beta S^* I^{(M)} - \gamma I^*+ \mu IN^* - \nu I^* -\delta I^*)dt +(\sigma_\beta S^* I^* + \sigma_\mu I I^*)dW,\\
	dR^* &= (\gamma I^* + \mu R N^* - \nu R^* - \delta R^*)dt +\sigma_\mu RR^* dW. 
	\end{split}
 \end{equation}
 From a modeling perspective, one may prefer to replace each of $(1-\alpha)^2SI, S^* I^*$ in the stochastic terms with $(1-\alpha)SI^{(M)}, S^* I^{(M)}$ (and similarly for stochastic terms involving $\sigma_\mu$). We make these particular modeling decisions so that that stochastic term in each equation is proportional to the population itself, which is necessary in \eqref{eq:bound3} in the proof of theorem \ref{t.w09271} below. We note that, besides the proof of theorem \ref{t.w09271}, this choice has no effect, so that theorems of section \ref{sec:stoch2} still hold for the full system where $(1-\alpha)^2SI, S^* I^*$ in the stochastic terms are replaced with $(1-\alpha)SI^{(M)}, S^* I^{(M)}$, with the caveat that one must assume global existence of a solution which is eventually nonnegative. This is because in the proof of theorems~\ref{thm:E0expMSS} and \ref{thm:WorstCase}, the only bound that we use for the stochastic terms is the total population bound $\tfrac b \delta$ which works just as well for the full stochastic terms as for the reduced ones. Further, with or without this modification, system \eqref{eq:StochSIR} captures the same behavior in the asymptotic regimes studied in section \ref{sec:stoch2} since those regimes alternately involve the whole population becoming compliant so that $(1-\alpha)SI^{(M)} \to (1-\alpha)^2SI$, the whole population becoming noncompliant so that $S^*I^{(M)} \to S^*I^*$, and similar observations can be made regarding the stochastic terms involving spread of noncompliance.
 
 We note that adding together all equations from \eqref{eq:StochSIR} shows that the total population $N_{\text{total}} = S+I+R+S^*+I^*+R^*$ is deterministic and satisfies the same bound derived above from \eqref{eq:totalPop} and \eqref{eq:totalPopBound}. For our purposes, we will be considering positive, \emph{deterministic} initial data $(S_0, I_0, R_0, S^*_0, I^*_0, R^*_0)$. The assumption that initial data is deterministic is necessary for some of stability results below. Given this, we have the following result. 

\begin{theorem}\label{t.w09271}
    For any positive initial data $(S_0, I_0, R_0, S^*_0, I^*_0, R^*_0)$, there exists a unique, globally-defined, nonnegative solution $(S, I, R, S^*, I^*, R^*)$ of system \eqref{eq:StochSIR} with probability 1.
\end{theorem}

\begin{proof}
    The proof uses the same basic strategy as \cite{Gray,JJS}.
    
    Since all terms in \eqref{eq:StochSIR} are locally Lipschitz with respect to $(S,I,R,S^*,I^*,R^*)$, given  positive initial data, a local solution of \eqref{eq:StochSIR} exists on a time interval $[0,\tau_E)$. We'd like to prove that this solution remains bounded (hence exists globally) and nonnegative. To do so, starting with $k_0 \in \N$ such that $$\max\{S_0,I_0,R_0,S_0^*,I_0^*,R_0^*\} \le k_0, \,\,\,\,\, \min\{S_0,I_0,R_0,S_0^*,I_0^*,R_0^*\} \ge \frac 1 {k_0},$$ we define $$\tau_k = \inf\left\{t \in [0,\tau_E) : \max\{S,I,R,S^*,I^*,R^*\} \ge k \,\,\, \text{ or } \,\,\, \min\{S,I,R,S^*,I^*,R^*\} \le \tfrac 1 k\right\}$$ for all $k \ge k_0.$ Then $\tau_k$ is an increasing sequence. Further if $\tau_k \to \infty$ almost surely as $k\to \infty$, then the solution almost surely exists and remains nonnegative for all time. 

    Suppose that it is not true that $\tau_k\to \infty$ almost surely. Then \begin{equation} \label{eq:contraAssump}
    \exists T > 0, \eps \in (0,1), k^* \ge k_0 \,\,\, \text{ such that } \,\, P\{\tau_k \le T \} \ge \eps \,\, \text{ for } \, k \ge k^*.
    \end{equation} We let $\overline N$ be a bound on the total population and fix $k \ge k^*$. Since all populations are nonnegative up to time $\tau_k$, each of the subpopulations is also bounded by $\overline N$.  

    Define $$V(t) = \sum_{A \in \{S,I,R,S^*,I^*,R^*\}} (A(t)-1-\log A(t)).$$ From the inequality $a - 1\ge \log (a)$ for $a > 0$, we see that $V$ is nonnegative. Further $V \to \infty$ if any of the populations blows up or goes to zero. Abstractly, we write each equation from \eqref{eq:StochSIR} as $$dA = F_Adt + \Sigma_A dW,$$ where $F_A$ represents the underlying dynamics and $\Sigma_A$ represents the stochastic drift. By It\^o's formula \cite[Ch. 2]{pardoux}, we have \begin{equation} \label{eq:dV} dV = \sum_{A \in \{S,I,R,S^*,I^*,R^*\}} \left[\left(1 - \frac{1}{A}\right)F_A + \frac{\Sigma_A^2}{2A^2}\right]dt + \left(1-\frac 1 A\right)\Sigma_AdW.\end{equation}

    Operating on the time interval $[0,\tau_k)$, we now simplify and bound \eqref{eq:dV} term-by-term, remembering that populations are nonnegative on this time interval. In what follows $C$ is a positive constant that changes from line-to-line and depends on the ambient parameters $b,\delta,\alpha,\beta,\gamma,\mu,\nu.$  Looking at \eqref{eq:StochSIR}, we have \begin{equation} \label{eq:bound1}
    \begin{split}
        \sum_{A \in \{S,I,R,S^*,I^*,R^*\}} \left(1-\frac{1}{A}\right)F_A &= b-\delta N_{\text{total}} \\ &\hspace{6mm}+ \left(-\frac{(1-\xi)b}{S} + \beta (1-\alpha)I^{(M)} + \mu N^* - \nu \frac{S^*}{S} + \delta \right)\\ 
        &\hspace{6mm}+\left(-\frac{\beta(1-\alpha)SI^{(M)}}{I} + \gamma + \mu N^* - \nu\frac{I^*}{I} + \delta\right) \\ 
        &\hspace{6mm}+\left(-\gamma \frac I R + \mu N^* -\nu \frac{R^*}{R} + \delta \right) \\
        &\hspace{6mm}+ \left(-\frac{\xi b}{S^*} + \beta I^{(M)} -\mu \frac{SN^*}{S^*} + \nu + \delta \right)\\
        &\hspace{6mm}+ \left( -\beta\frac{S^*I^{(M)}}{I^*} + \gamma - \mu\frac{IN^*}{I^*} + \nu + \delta \right)\\
        &\hspace{6mm}+ \left(-\gamma\frac{I^*}{R^*} - \mu \frac{RN^*}{R^*} + \nu + \delta\right) \\
        &\le b + \beta(2-\alpha)I^{(M)} + 2 \gamma +3\mu N^* +3\nu + 6 \delta \le C(1+\overline N).
    \end{split}
\end{equation} Next, \begin{equation} \label{eq:bound2}
\begin{split}
\sum_{A \in \{S,I,R,S^*,I^*,R^*\}} \frac{\Sigma_A^2}{2A^2} 
&= \frac 1 2 \bigg[ (-\sigma_\beta (1-\alpha)I - \sigma_\mu S^*)^2 + (\sigma_\beta(1-\alpha)S - \sigma_\mu I^*)^2 +\sigma_\mu^2(R^*)^2 \\
&\hspace{15mm}+ (-\sigma_\beta I^* + \sigma_\mu S)^2 + (\sigma_\beta S^* + \sigma_\mu I)^2 + \sigma_\mu^2 R^2\bigg] \\
&\le C\overline N^2.
\end{split}
\end{equation} And lastly, \begin{equation}
    \label{eq:bound3}\begin{split}
        \sum_{A \in \{S,I,R,S^*,I^*,R^*\}} \left(1-\frac 1 A\right)\Sigma_A &= \left(1-\frac 1 S\right)(-\sigma_\beta (1-\alpha)SI -\sigma_\mu SS^*) \\
        &\hspace{6mm}+ \left(1-\frac 1 I\right)(\sigma_\beta(1-\alpha) SI - \sigma_\mu II^*) \\ 
        &\hspace{6mm} +\left(1-\frac 1 {S^*}\right)(-\sigma_\beta S^*I^* + \sigma_\mu SS^*) \\
        &\hspace{6mm}+ \left(1 - \frac{1}{I^*}\right) (\sigma_\beta S^*I^* + \sigma_\mu II^*) \\
        &\hspace{6mm} + \left(1-\frac 1 R\right)(-\sigma_\mu RR^*)+ \left(1-\frac1{R^*}\right)\sigma_\mu RR^*\\
        &= \sigma_\beta((1-\alpha)(I-S) + I^*-S^*) \\ &\hspace{6mm}+ \sigma_\mu(S^*+I^*+R^* - S - I -R)\\ 
        &\backdefeq f(S,I,R,S^*,I^*,R^*).
    \end{split}
\end{equation} Inserting \eqref{eq:bound1}-\eqref{eq:bound3} into \eqref{eq:dV}, we see \begin{equation} \label{eq:dVbound}
dV \le C\left(1+\overline N + \overline N^2\right)dt + L(S,I,R,S^*,I^*,R^*)dW
\end{equation} where $f(S,I,R,S^*,I^*,R^*)$ is a linear function as defined at the end of \eqref{eq:bound3}. For brevity, we set $M = C(1+\overline N + \overline N^2)$.

For any $s \le T$, set $t = \min(\tau_k,s)$, \eqref{eq:dVbound} yields \begin{equation}\label{eq:dVintbound1}
    \int^t_0 dV \le \int^t_0 M dt + \int^t_0 f(S,I,R,S^*,I^*,R^*)dW.
\end{equation} Since $f$ is bounded for times in $[0,\tau_k)$, in expectation the stochastic integral disappears and we have \begin{equation} \label{eq:Vexbound}
    \E\Big[V(t)\Big] \le V(0) + MT. 
\end{equation} By the assumption that $(S_0,I_0,R_0,S^*_0,I^*_0,R^*_0)$ are positive constants, we see that $V(0)$ is a positive constant. Finally, define $\Omega_k = \{\tau_k \le T\}$, so that $P(\Omega_k) \ge \eps$ by our assumption \eqref{eq:contraAssump}. We see that for any $\omega \in \Omega_k$, one of the subpopulations must take value $k$ or $1/k$ at time $\tau_k$ which means that \begin{equation} \label{eq:lowerVbound}
V(\tau_k;\omega) \ge \min\left(k-1-\log k, \frac 1 k - 1 + \log k\right)
\end{equation} But then \eqref{eq:Vexbound} and $P(\Omega_k) \ge \eps$ results in \begin{equation} \label{eq:contraBound} \begin{split} V(0) + MT &\ge \E\Big[ 1_{\Omega_k} V(\tau_k;\omega)  \Big] \ge \min\left(k-1-\log k, \frac 1 k - 1 + \log k\right). \end{split} \end{equation} Sending $k \to \infty$ gives a contradiction. Thus \eqref{eq:contraAssump} is false, and $\lim_{k\to\infty} \tau_k =\infty$ almost surely, proving that solutions exist and remain nonnegative globally-in-time with probability 1. \end{proof} 

With global existence and positivity in hand, we turn to large-time asymptotic behavior of \eqref{eq:StochSIR}.

\section{Asymptotic behavior of the stochastic system near disease free states} \label{sec:stoch2}

{\CPEDIT In this section, we perform some asymptotic analysis of \eqref{eq:StochSIR} with respect to certain disease free states. In particular, we prove stability results for \eqref{eq:StochSIR} analogous to \ref{eq:stabilityOfDFE}(i) and the ``worst case scenario" DFE addressed as part of theorem \ref{eq:stabilityOfDFE}(iii) and discussed in remark \ref{rem:Thm25}. Besides these, we would like to say something about the behavior of solutions to \eqref{eq:StochSIR} near the disease free states $X_{s,s^*} = (s,0,0,s^*,0,0)$ where $s,s^*$ are the non-zero elements of $x_3$ as defined in \eqref{eq:DFE3}. These are not equilibria for the stochastic model but we would like to quantify how far the stochastic drift can push the dynamics away from these states if the stability conditions of theorem \ref{eq:stabilityOfDFE}(ii) or \ref{eq:stabilityOfDFE}(iii) are met. }

Our first two stability results follow from lemma \ref{lem:expMSS} which is a special case of Theorem 4.4 in \cite[Ch. 4]{MaoBook}. Before introducing the lemma, we establish some preliminary notation. Consider the stochastic ODE in $d$-dimensions \begin{equation} \label{eq:stochODEGeneral}
    dX = F(X)dt + G(X)dW
\end{equation} where $F,G:\R^d \to \R^d$ are twice-differentiable and $W$ is a one-dimensional Wiener process. In the case where $F(0) = G(0) = 0$, we say that $X \equiv 0$ is the trivial solution of \eqref{eq:stochODEGeneral}. We define the operator $L$ associated with \eqref{eq:stochODEGeneral} by its action on any smooth function $V(x)$: $$LV(x) = F(x)\cdot \nabla V(x) + \tfrac 1 2 G(x) \cdot D^2V(x) G(x).$$ By It\^o's formula \cite[Ch. 2]{pardoux}, assuming $X$ satisfies \eqref{eq:stochODEGeneral}, we have $$dV(X) = LV(X)dt + G(X)\cdot \nabla V(X) dW.$$ Because of this, we have Lyapunov stability results for stochastic ODE using conditions on $LV(x)$ that are analogous to the results from deterministic ODE which depend on $\tfrac d {dt}V(x(t)).$ Specifically, we use the following result, which holds when the initial data is deterministic, which we assumed for our system above.

\begin{lemma}  \label{lem:expMSS} \cite[Ch. 4, Thm. 4.4]{MaoBook} Assume that there is a non-negative function $V \in C^{2}(R^d)$ and positive constants $A,B,C$, such that
$$
A|x|^2 \leq V(x) \leq B|x|^2 \quad \text { and } LV(x) \leq -C V(x)
$$
for all $x \in \R^d$. Then the trivial solution of \eqref{eq:stochODEGeneral} is exponentially mean-square stable. Specifically, if $X$ satisfies \eqref{eq:stochODEGeneral}, we have $$\mathbb E\big(\abs{X}^2\big) \le \frac B A \abs{X_0}^2 e^{-Ct}$$
\end{lemma}

For the following theorem, we assume that $\xi = 0$ and define a disease extinction threshold corresponding to the fully compliant disease free equilibrium $X_1 = (\frac{b}{\delta}, 0,0,0,0,0)$ by
\begin{equation}\label{eq:stochR01}
    \mathscr R_0^\sigma(\tfrac b \delta,0) =\frac{\beta\Big(\tfrac b \delta \Big)(1-\alpha)^2 + \tfrac{\sigma^2_\beta}2\Big(\tfrac b \delta\Big)^2(1-\alpha)^4}{\gamma+\delta}.
\end{equation}

\begin{theorem} \label{thm:E0expMSS}
      Suppose $\xi = 0$. If $$\frac{b}{\delta} + \frac{\sigma^2_\mu}{2\mu}\left(\frac b \delta \right)^2 < \frac{\nu+\delta}{\mu} \,\,\,\,\, \text{ and } \,\,\,\,\, \mathscr R^\sigma_0(\tfrac b \delta,0) < 1,$$ then the solution  $X_1 = (\tfrac b \delta,0,0,0,0,0)$ of system \eqref{eq:StochSIR} is exponentially mean-square stable.
\end{theorem}

\begin{remark} \label{rem:Thm42}
As alluded to above, the proof follows by establishing a stochastic Lyapunov function. Our function below was chosen to remain positive-definite while eliminating many of the nonlinear terms for easier calculation of $LV$, but other choices would likely work. Before embarking on the proof, we draw attention to the formal similarity between theorem \ref{eq:stabilityOfDFE}(i) and theorem \ref{thm:E0expMSS}. For the fully compliant disease free state to be stable, we need two conditions: one which guarantees that noncompliance does not spread too fast, and one which guarantees that the disease does not spread too fast. In this way, the theorems say the exact same thing regarding sufficient conditions for stability, but the conditions for theorem \ref{thm:E0expMSS} are strengthened to account for the random perturbation. Specifically, $\mathscr R_0(\tfrac b \delta,0)$ as defined in \eqref{eq:bestWorst} is strictly smaller than $\mathscr R^\sigma_0(\tfrac b \delta,0)$ as defined in \eqref{eq:stochR01} so that the condition $\mathscr R^\sigma_0(\tfrac b \delta,0)<1$ of theorem \ref{thm:E0expMSS} is strictly stronger than the condition $\mathscr R_0 (\tfrac b \delta,0)<1$ of theorem \ref{eq:stabilityOfDFE}(i). Similarly the assumption $\frac{b}{\delta} + \frac{\sigma^2_\mu}{2\mu}\left(\frac b \delta \right)^2 < \frac{\nu+\delta}{\mu}$ of theorem \ref{thm:E0expMSS} is strictly stronger than the assumption $\tfrac b \delta < \tfrac {\nu+\delta}\mu$ of theorem \ref{eq:stabilityOfDFE}(i). One final note is that if we set $\sigma_\beta = \sigma_\mu = 0$ in the stability conditions of theorem \ref{thm:E0expMSS}, we directly recover the stability conditions of theorem \ref{eq:stabilityOfDFE}(i), so theorem \ref{thm:E0expMSS} serves as a direct generalization of theorem \ref{eq:stabilityOfDFE} to the stochastic setting, with ``local asymptotic stability" replaced by ``exponential mean-square stability." In particular, we get a \emph{stronger} version of the sufficient condition for stability in theorem \ref{eq:stabilityOfDFE}(i), since this result is global rather than local. \end{remark}

\begin{proof}
    Suppose $(S,I,R,S^*,I^*,R^*)$ is a solution of \eqref{eq:StochSIR}. 
    
    We define new variables $u=S-\frac{b}{\delta}, v=I, w=R, u^*=S^*, v^*=I^*, w^*=R^*$. Thus, $v^{(M)} = (1-\alpha)I+I^*= (1-\alpha)v + v^*$ and $n^*=u^*+v^*+w^*$. Note that by nonnegativity guaranteed by theorem \ref{t.w09271} as well as the total population bound of $b/\delta$, we have that $u \le 0$, $v,w,u^*,v^*,w^* \ge 0$. Further the sum of any subcollection of the populations will be bounded by $b/\delta$ (e.g. $\left(u+\frac b\delta\right)  + v + w \le \frac b \delta$ and $u^* + v^* + w^* \le \frac b \delta$). In what follows, we use this total population bound several times without stopping to mention it.  

    The new variables $x = (u,v,w,u^*,v^*,w^*)$ satisfy 
    \begin{equation}  \label{eq:StochSIRNewVars}
\begin{split}
	du &= (- \beta(1-\alpha)(u+\tfrac{b}{\delta})v^{(M)} -\mu (u+\tfrac{b}{\delta})n^* + \nu u^* -\delta u)dt
    \\ &\hspace{3cm}+
    (-\sigma_\beta(1-\alpha)^2v(u+\tfrac{b}{\delta})) -\sigma_\mu (u+\tfrac{b}{\delta})u^*)dW\\
	dv &= (\beta(1-\alpha)(u+\tfrac{b}{\delta})v^{(M)} - \gamma v -\mu v n^* + \nu v^* - \delta v)dt 
    \\ &\hspace{3cm}+
    (\sigma_\beta(1-\alpha)^2v(u+\tfrac{b}{\delta}) -\sigma_\mu v v^*)dW, \\
	dw &= (\gamma v - \mu w n^* + \nu w^*- \delta w)dt - \sigma_\mu w w^* dW, \\
    du^* &= ( - \beta u^*v^{(M)} + \mu (u+\tfrac{b}{\delta}) n^* - \nu u^* -\delta u^*)dt +(-\sigma_\beta u^* v^* + \sigma_\mu (u+\tfrac{b}{\delta}) u^*)dW,\\
	dv^* &= (\beta u^* v^{(M)} - \gamma v^*+ \mu v n^* - \nu v^* -\delta v^*)dt +(\sigma_\beta u^* v^* + \sigma_\mu v v^*)dW,\\
	dw^* &= (\gamma v^* + \mu w n^* - \nu w^*- \delta w^*)dt +\sigma_\mu w w^* dW , 
	\end{split}
 \end{equation}
 We want to prove that the trivial solution of \eqref{eq:StochSIRNewVars} is exponentially mean-square stable. To this end, let $L$ be the generating operator of the system \eqref{eq:StochSIRNewVars}. We define the stochastic Lyapunov function \begin{equation} \label{eq:stochLyapunovFunc} \begin{split} V(x) &= (u+v+u^*+v^*)^2 + c_2(n^*)^2 +c_3(v+v^*)^2 + (v^*)^2 + (w+w^*)^2+(w^*)^2 \end{split}\end{equation} where $c_2,c_3$ are positive constants to be chosen later. Recalling that $n^* = u^*+v^*+w^*$, it is clear that $V$ is a positive definite quadratic form, from which the bounds \begin{equation} \label{eq:squareBound}A\abs{x}^2 \le V(x) \le B\abs{x}^2 \end{equation}follow easily (here $A$ and $B$ are positive constants defined in terms of $c_2,c_3$). Our result now follows from lemma \ref{lem:expMSS} if we can achieve a bound of the form $LV(x) \le -CV(x)$. For this bound, by \eqref{eq:squareBound}, it is sufficient to prove that \begin{equation}
     \label{eq:LVbound} LV(x) \le - C\abs{x}^2 = -C(u^2+v^2+w^2+(u^*)^2 + (v^*)^2 + (w^*)^2),
 \end{equation} for a positive constant $C$. We compute \begin{equation}\label{eq:bigUglyEquation1} \begin{split}
LV(x) &= -2\delta(u+v+u^*+v^*)^2 - 2\gamma(v+v^*)(u+v+u^*+v^*) \\ 
&\hspace{0.5cm}+2c_2\Big(\mu\Big(u+\tfrac b \delta + v + w\Big) - (\nu + \delta)\Big)(n^*)^2 + c_2 \sigma_\mu^2\Big(\Big(u+\tfrac b \delta\Big)u^* + vv^*+ww^*\Big)^2\\
&\hspace{0.5cm}+2c_3(v+v^*)\Big(\beta\Big((1-\alpha)\Big(u+\tfrac b \delta\Big) + u^*\Big)((1-\alpha)v+v^*)  - (\gamma + \delta)(v+v^*)\Big) \\ 
&\hspace{0.5cm}\hspace{2in}+c_3\sigma_\beta^2\left((1-\alpha)^2\left(u+\tfrac b \delta\right)v + u^*v^*\right)^2\\
&\hspace{0.5cm}+ 2 v^*\Big(\beta u^*((1-\alpha)v + v^*) + \mu v n^* - (\gamma+\nu+\delta)v^*\Big) + c_4(\sigma_\beta u^*v^* +\sigma_\mu vv^*)^2\\
&\hspace{0.5cm}+2(w+w^*)\Big( \gamma(v+v^*) - \delta (w+w^*)\Big) \\
&\hspace{0.5cm}+2 w^*\Big(\gamma v^* - \delta w^* +\mu wn^* - \nu w^* \Big) + \sigma_\mu^2 w^2(w^*)^2\\
&\defeq \ell_1 + \ell_2 + \ell_3 + \ell_4 + \ell_5 + \ell_6,
 \end{split}
 \end{equation} where $\ell_i$ is the $i^{\text{th}}$ line in the long equation of \eqref{eq:bigUglyEquation1} (considering the indented line as belonging to the line which precedes it). We look for quadratic bounds on each $\ell_i$ keeping in mind that $u\le 0$ and all other variables are nonnegative. First we have \begin{equation} \label{eq:line1bound} \begin{split}
     \ell_1 &\le -2\delta(u^2+v^2+(u^*)^2+(v^*)^2) - 4\delta u u^* -2\gamma(v^2 + (v^*)^2) - 2(\gamma+2\delta) u(v+v^*)\\
     &\le -2\delta(u^2+v^2+(u^*)^2+(v^*)^2) + \tfrac{\delta}2u^2 + 8\delta (u^*)^2 + \tfrac{\delta}2u^2 + \tfrac{2(\gamma+2\delta)^2}{\delta}v^2 \\ &\hspace{3in}+ \tfrac{\delta}2u^2 + \tfrac{2(\gamma+2\delta)^2}{\delta}(v^*)^2\\
     &\le -\tfrac \delta 2 u^2 -\Big(2\delta - \tfrac{2(\gamma+2\delta)^2}{\delta}\Big) v^2 + 6\delta (u^*)^2 -\Big(2\delta - \tfrac{2(\gamma+2\delta)^2}{\delta}\Big) (v^*)^2.
     \end{split}
 \end{equation} For $\ell_2$, we have \begin{equation}\label{eq:line2boundPre}
 \begin{split}
    \ell_2 &\le 2c_2\Big(\mu\Big(\tfrac { b}\delta \Big)- (\nu + \delta)\Big)(n^*)^2 +c_2 \sigma_\mu^2\Big(\tfrac b \delta \Big)^2(n^*)^2\\
    &= 2c_2\mu\Big(\tfrac { b}\delta  + \tfrac{\sigma_\mu^2}{2\mu} \Big(\tfrac b \delta \Big)^2 -\tfrac{\nu+\delta}\mu\Big)(n^*)^2.
 \end{split} 
 \end{equation} With the assumption $\frac{b}{\delta} + \frac{\sigma^2_\mu}{2\mu}\left(\frac b \delta \right)^2 < \frac{\nu+\delta}{\mu}$, the coefficient above is negative, so we can write \begin{equation} \label{eq:line2bound}
     \ell_2 \le -2\tilde c_2 (n^*)^2 \le -2\tilde c_2((u^*)^2 + (v^*)^2 + (w^*)^2)
  \end{equation} 
 for a positive constant $\tilde c_2$ which is proportional to the original constant $c_2.$ Because $\ell_2$ provides a negative coefficient in front of $(n^*)^2$ which can be made arbitrarily large, we can be entirely cavalier with most terms involving $u^*,v^*,w^*$ moving forward; the only terms which need to be handled carefully are those which include $v$.
 
 Continuing to $\ell_3$, the difficult question is what to do with the triple term $$\tilde \ell_3 = (v+v^*)\Big((1-\alpha)\Big(u+\tfrac b \delta\Big) + u^*\Big)((1-\alpha)v+v^*).$$ We first write $(1-\alpha)(u+\tfrac b \delta) + u^* = (1-\alpha)(u+\tfrac b\delta  +u^*) + \alpha u^* \le (1-\alpha)\tfrac b \delta + \alpha u^*$ to see \begin{equation}\label{eq:line3boundPre} \begin{split}
 \tilde \ell_3 &\le \tfrac b \delta (1-\alpha)(v+v^*)((1-\alpha)v+v^*) + \alpha u^* (v+v^*)((1-\alpha)v+v^*)\\
 &\le \tfrac{b}{\delta}(1-\alpha)((1-\alpha)v^2 + (2-\alpha)vv^* + (v^*)^2) + \tfrac {\alpha b} \delta ((1-\alpha)u^*v+u^*v^*).
 \end{split} 
 \end{equation} Fixing $\eps>0$, we see $(2-\alpha)vv^* \le \tfrac \eps 2v^2 + \frac{(2-\alpha)^2}{2\eps}(v^*)^2$ and $\alpha(1-\alpha)u^*v \le \tfrac \eps 2v^2 + \frac{\alpha^2(1-\alpha)^2}{2\eps}(u^*)^2.$ Plugging these into \eqref{eq:line3boundPre} (and bounding $u^*v^*$ similarly), we arrive at \begin{equation} \label{eq:line3boundPre2}
    \tilde \ell_3 \le \tfrac b \delta ((1-\alpha)^2 + \eps)v^2 + \tfrac C \eps (u^*)^2 + \tfrac C \eps (v^*)^2
 \end{equation} for a positive constant $C$ comprised of the ambient parameters (in what follows, the positive constant $C$ will change from line to line). With this in hand, we see 
 \begin{equation} \label{eq:line3boundPre3}
     \begin{split}
    \ell_3 &\le 2c_3\bigg(\tfrac b \delta \beta ((1-\alpha)^2 + \eps)v^2 + \tfrac C \eps (u^*)^2 + \tfrac C \eps (v^*)^2\bigg)  - 2c_3 (\gamma+\delta)(v+v^*)^2 \\ &\hspace{6.5cm} + c_3\sigma_\beta^2\left((1-\alpha)^2\left(u+\tfrac b \delta\right)v + u^*v^*\right)^2     \\
    &\le 2c_3\bigg(\tfrac b \delta\beta ((1-\alpha)^2 + \eps)v^2 + \tfrac C \eps (u^*)^2 + \tfrac C \eps (v^*)^2\bigg)-2c_3(\gamma+\delta)(v^2 + (v^*)^2) \\ &\hspace{6.5cm} + c_3\sigma_\beta^2\Big(\tfrac b  \delta\Big)^2 \left((1-\alpha)^2 v + v^*\right)^2. 
     \end{split}
 \end{equation} Finally, we use the same type of bound on $((1-\alpha)^2v^2 + v^*)^2$: after multiplying it out, we bound the cross term $2(1-\alpha)^2vv^* \le \eps v^2 + \tfrac C\eps (v^*)^2.$ Putting all this together, we arrive at 
\begin{equation} \label{eq:line3bound} \begin{split}
    \ell_3 &\le 2c_3\bigg(\tfrac b \delta \beta ((1-\alpha)^2+\eps) + \tfrac{\sigma^2_\beta}2\Big(\tfrac b \delta\Big)^2((1-\alpha)^4+\eps) - (\gamma+\delta)\bigg)v^2 + \tfrac C \eps (u^*)^2 + \tfrac C \eps (v^*)^2.
    \end{split}
\end{equation}  

Next, the bound for $\ell_4$ is quite simple: \begin{equation}\label{eq:line4bound} \begin{split}
\ell_4 &= 2 v^*\Big(\beta u^*((1-\alpha)v + v^*) + \mu v n^* - (\gamma+\nu+\delta)v^*\Big) + (\sigma_\beta u^*v^* +\sigma_\mu vv^*)^2\\
&\le \frac{2\beta b}{\delta} u^*v^* + \tfrac{2\mu b}{\delta}v^*n^* - 2(\gamma+\nu+\delta)(v^*)^2 + \Big(\tfrac b \delta \Big)^2 (\sigma_\beta + \sigma_\mu)^2 (v^*)^2\\
&\le C((u^*)^2 + (v^*)^2 + (w^*)^2).
\end{split}
\end{equation} For $\ell_5$, we see \begin{equation} \label{eq:line5bound}
\begin{split}
\ell_5 &= 2\gamma(w+w^*)(v+v^*) - 2\delta(w+w^*)^2\\
&\le 2\gamma vw + 2\gamma v^*w + 2\gamma vw^* + 2\gamma v^*w^*  -2\delta(w^2 + (w^*)^2)\\
&\le \tfrac{2\gamma^2}{\delta}v^2 +\tfrac{\delta}{2}w^2 + \tfrac{2\gamma^2}{\delta}(v^*)^2 +\tfrac{\delta}{2}w^2  - 2\delta w^2 + C((v^*)^2 + (w^*)^2) \\
&\le \tfrac{2\gamma^2}{\delta}v^2 - \delta w^2 + C((v^*)^2 + (w^*)^2).
\end{split}
\end{equation} And finally, for $\ell_6$, we simply use $w \le b/\delta$ to arrive at \begin{equation} \label{eq:line6bound} \begin{split}
\ell_6 &\le 2(\gamma v^*w^* - (\delta+\nu)(w^*)^2 + \mu \Big(\tfrac b\delta\Big)w^*n^*) + \sigma_\mu^2 \Big(\tfrac b \delta \Big)^2 (w^*)^2 \\ &\le C((u^*)^2 + (v^*)^2 + (w^*)^2). \end{split}
\end{equation} Adding all the bounds \eqref{eq:line1bound}, \eqref{eq:line2bound}, \eqref{eq:line3bound}, \eqref{eq:line4bound}, \eqref{eq:line5bound}, \eqref{eq:line6bound} together, we see that \begin{equation} \label{eq:FinalBound} \begin{split} 
LV(x)&\le -\tfrac \delta 2 u^2 - \delta w^2 + \left(C\Big(1+\tfrac 1 \eps\Big) - 2\tilde c_2\right)((u^*)^2 + (v^*)^2 + (w^*)^2) \\
&\hspace{1cm} +{\scriptstyle\left[2c_3\bigg(\tfrac b \delta \beta((1-\alpha)^2+\eps) + \tfrac{\sigma^2_\beta}2\Big(\tfrac b \delta\Big)^2((1-\alpha)^4+\eps) - (\gamma+\delta)\bigg) +\tfrac{2\gamma^2 +2(\gamma+2\delta)^2}{\delta} - 2\delta\right]}v^2 \\
&= -\tfrac \delta 2 u^2 - \delta w^2 + \left(C\Big(1+\tfrac 1 \eps\Big) - 2\tilde c_2\right)((u^*)^2 + (v^*)^2 + (w^*)^2) \\
&\hspace{1cm}+ \left[2c_3(\gamma+\delta)\bigg(\mathscr R^\sigma_0(\tfrac b\delta,0) + \bigg(\tfrac b \delta + \tfrac{\sigma^2_\beta}2\Big(\tfrac b \delta\Big)^2\bigg)\tfrac{\eps}{\gamma+\delta} - 1\bigg) +\tfrac{2\gamma^2 +2(\gamma+2\delta)^2}{\delta} - 2\delta  \right]v^2.
\end{split}
\end{equation} Assuming $\mathscr R^\sigma_0(\tfrac b\delta,0) < 1$, we can choose $\eps$ small enough that $\mathscr R^\sigma_0(\tfrac b\delta,0) + \bigg(\tfrac b \delta + \tfrac{\sigma^2_\beta}2\Big(\tfrac b \delta\Big)^2\bigg)\tfrac{\eps}{\gamma+\delta} < 1$. Having done so, we define $$c_3 = \frac{\gamma^2 + (\gamma+2\delta)^2}{\delta(\gamma+\delta)\left(1 - \mathscr R^\sigma_0(\tfrac b\delta,0) - \left(\tfrac b \delta + \tfrac{\sigma^2_\beta}2\Big(\tfrac b \delta\Big)^2\right)\tfrac{\eps}{\gamma+\delta}\right)} > 0.$$ This choice of $c_3$ causes the first two terms in the brackets of \eqref{eq:FinalBound} to cancel, leaving \begin{equation}
    \label{eq:FinalBound2}
    LV(x) \le -\tfrac \delta 2 u^2  -2\delta v^2- \delta w^2 + \left(C\Big(1+\tfrac 1 \eps\Big) - 2\tilde c_2\right)((u^*)^2 + (v^*)^2 + (w^*)^2).
\end{equation} Finally, with $\eps$ fixed, we can take $\tilde c_2$ large enough that $C\Big(1+\tfrac 1 \eps\Big) - 2\tilde c_2<0$ whereupon we achieve \begin{equation}
     \label{eq:LVboundAgain} LV(x) \le  -C(u^2+v^2+w^2+(u^*)^2 + (v^*)^2 + (w^*)^2),
 \end{equation} and applying lemma \ref{lem:expMSS} shows that the DFE $X_1 = (\tfrac b \delta, 0,0,0,0,0)$ is exponentially mean-square stable. 
\end{proof}

Next we turn our attention to worst case scenario of $\xi = 1$ and $\nu= 0.$ In this case, the DFE is $X_2 = (0,0,0,\tfrac b \delta,0,0)$ and we define a disease extinction threshold by \begin{equation}
    \label{eq:stochR0WorstCase}
    \mathscr R^\sigma_0(0,\tfrac b \delta) = \frac{\beta\Big(\tfrac b \delta\Big) + \frac{\sigma_\beta^2}{2}\Big(\tfrac b \delta \Big)^2}{\gamma +\delta}
\end{equation}

\begin{theorem} \label{thm:WorstCase}
    Suppose $\xi = 1$ and $\nu = 0.$ If $\mathscr R^\sigma_0(0,\tfrac b \delta) < 1,$ then for any solution $(S,I,R,S^*,I^*,R^*)$ of \eqref{eq:StochSIR}, we have \begin{equation} \label{eq:expDecayIIs}\mathbb E(I^2), \mathbb E((I^*)^2) \le 2 \max\{I_0^2,(I_0^*)^2\} e^{-C t}, \,\,\,\,\,\,\,\, \text{ where } \,\, C = \frac{2\Big(1-\mathscr R^\sigma_0(0,\tfrac b \delta)\Big)}{(\gamma+\delta)} > 0. \end{equation} If in addition $\frac{\sigma_\mu^2}2\Big(\tfrac  b\delta\Big)^2 < \delta$, then $X_2$ exponentially mean-square stable. 
\end{theorem}

\begin{remark} \label{rem:Thm44}
    We split theorem \ref{thm:WorstCase} into two statements to reflect what might actually be important to a policy-maker. Here $X_2$ could be unstable in two different ways: (1) infections could increase and persist or (2) the stochasticity in transmission of noncompliance could cause random spikes in compliance to occur. For the policy-maker, scenario (2) could be seen as a benefit, so there is no reason to actually hope that $X_2$ is stable with respect to this sort of perturbation. However, one \emph{would} want to guarantee that scenario (1) does not occur. The first statement, regarding decay of second moments of $I$ and $I^*$, gives conditions under which the disease is expected to die out. The second statement regarding mean-square stability is included for mathematical completeness, though its unclear if this is even desirable in a real-world scenario. The condition $\frac{\sigma_\mu^2}2\Big(\tfrac b \delta\Big)^2 < \delta$ in the second statement guarantees that the compliant population does not randomly spike. As seen in \eqref{eq:uvwBound} below, this could be relaxed to $\frac{\sigma_\mu^2}2\Big(\tfrac b \delta\Big)^2 < \delta + \eta \mu $ where $\eta$ is any lower bound on the total noncompliant population. Regardless, setting $\sigma_\beta = \sigma_\mu = 0$, the latter conditions reduces to $\delta >0,$ and $\mathscr R^\sigma_0(0,\tfrac b \delta)$ reduces to $\mathscr R_0(0,\tfrac b\delta)$ as defined in \eqref{eq:bestWorst}, and thus theorem \ref{thm:WorstCase} reduces to the sufficient stability condition of theorem \ref{eq:stabilityOfDFE}(iii) with $\xi= 1$ and $\nu = 0$.  
\end{remark}

\begin{proof}
    To prove exponential decay of the second moments of $I$ and $I^*$ under the assumption that $\mathscr R_0^\sigma(0,\tfrac b\delta)<1$, we define $V = (I+I^*)^2$. Since $I,I^* \ge 0$, it is clear that  \begin{equation} \label{eq:Ubound0}I^2 + (I^*)^2 \le V \le 2(I^2 + (I^*)^2)\le 2\max\{I^2,(I^*)^2\}.\end{equation}  Then \begin{equation} \label{eq:Ubound1} \begin{split} 
    dV &= 2(I+I^*)\bigg(\beta((1-\alpha)S+S^*)((1-\alpha)I+I^*) - (\gamma+\delta)(I+I^*)  \\
    &\hspace{2cm}+\sigma_\beta^2((1-\alpha)^2SI + S^*I^*)^2\bigg)dt + 2\sigma_\beta(I+I^*)((1-\alpha)^2SI + S^*I^*)dW. 
    \end{split} 
    \end{equation} Letting $f=f(S,S^*,I,I^*)$ denote the stochastic term, and using $(1-\alpha)S+S^*\le \tfrac b \delta$ and $((1-\alpha)^2SI + S^*I^*) \le (S+S^*)(I+I^*)\le \tfrac b \delta (I+I^*)$, we have \begin{equation} \label{eq:boundVVV} \begin{split}
        dV &\le 2\bigg(\beta \Big(\tfrac{b}{\delta}\Big)+ \tfrac{\sigma_\beta^2}{2}\Big(\tfrac b \delta \Big)^2 - (\gamma+\delta)\bigg)(I+I^*)^2 + f\, dW \\ 
        &\le \frac{2}{\gamma+\delta}\Big(\mathscr R_0^\sigma(0,\tfrac b \delta) -1\Big)(I+I^*)^2dt + f\, dW\\
        &= -C(I+I^*)^2dt + f\, dW,
        \end{split}        
    \end{equation} with $C$ as defined in \eqref{eq:expDecayIIs}, which is negative due to the assumption that $\mathscr R_0^\sigma(0,\tfrac b \delta) <1$. The function $f$ is bounded, since its arguments remain bounded, and thus $\mathbb E\Big(\int^t_0 f dW\Big) = 0$, so integrating \eqref{eq:boundVVV}, taking the expectation and applying Fubini's theorem, we arrive at $$\mathbb E(V) \le V_0 - C \int^t_0 \mathbb E(V) dt.$$ Gronwall's inequality yields $$V \le V_0 e^{-Ct}$$ and we arrive at the desired exponential decay using \eqref{eq:Ubound0}.

    If in addition we assume that $\frac{\sigma_\mu^2}{2}\Big(\tfrac  b\delta\Big)^2 < \delta$, then we can prove $X_2 = (0,0,0,\tfrac b \delta,0,0)$ is exponentially mean-square stable by finding a suitable Lyapunov function in the same manner as in the proof of theorem \ref{thm:E0expMSS}. In this case, letting $u = S,v=I,w=R,u^* = S^*-\tfrac b \delta,v^*=I^*,w^*=R^*$, we define \begin{equation} \label{eq:stochLyapunovFunc2} V = (u+v+u^*+v^*)+c_2(u+v+w)^2+c_3(v+v^*)^2+v^2+(w+w^*)^2+w^2.\end{equation} Note, this is essentially the same Lyapunov function as in \eqref{eq:stochLyapunovFunc}, except all the asterisks have been flipped. It is still positive definite. Because the computation is so similar to that in the proof of theorem \ref{thm:E0expMSS}, we omit the full details but describe the strategy of the bounds. A key realization in this case (which is akin the the bound on $\ell_2$ in the proof of theorem \ref{thm:E0expMSS}), is that \begin{equation} \label{eq:uvwBound} d(u+v+w)^2 \le 2\bigg(\tfrac{\sigma_\mu^2}{2}\Big(\tfrac b \delta \Big)^2-\Big(\mu\Big(\big(u^*+\tfrac b \delta\big)+v^*+w^*\Big) + \delta\Big)\bigg)(u+v+w)^2 + (\cdots) dW. \end{equation} Since $\mu\Big(\big(u^*+\tfrac b \delta\big)+v^*+w^*\Big) \ge 0$ and $\tfrac{\sigma_\mu^2}{2}\Big(\tfrac b \delta \Big)^2<\delta$, the coefficient of $(u+v+w)^2$ is negative. Since this appears in \eqref{eq:stochLyapunovFunc2} with an arbitrary constant $c_2$, when computing $LV$, one has full control over any terms involving $u,v,w$, except those that also involve $v^*$. In what remains, the only problematic terms are cross terms like $vv^*$, but these are handled as above, except with the roles of $v$ and $v^*$ reversed so that $v$ absorbs any large constants: $vv^* \le \eps (v^{*})^2 + \frac C \eps v^2$. Doing this gives a bound akin to that in \eqref{eq:FinalBound}. In this case, we arrive at  \begin{align*}LV &= -C_1((u^*)^2+(w^*)^2) - \Big(\tilde c_2 - C_2\Big(1+\tfrac 1 \eps\Big)\Big)(u^2+v^2+w^2) \\
    &\hspace{1cm}+ \bigg(2c_3(\gamma+\delta)\Big(\mathscr R_0^\sigma(0,\tfrac b \delta) + C_3 \eps - 1\Big) + C_4 - 2\delta\bigg)(v^*)^2
    \end{align*} for some positive constants $C_1,C_2,C_3,C_4.$ Taking $\eps$ small enough, the coefficient in front of $c_3$ is negative, so we choose $c_3$ so as to cancel $C_4$, and then choose $\tilde c_2$ large enough that $\tilde c_2 - C_2 \Big(1+\tfrac 1 \eps\Big)>0$, and apply lemma \ref{lem:expMSS}, to show that $X_2 = (0,0,0,\tfrac b \delta,0,0)$ is exponentially mean-square stable. 
\end{proof} 

Finally, we discuss the mixed disease free states $X_{s,s^*} = (s,0,0,s^*,0,0)$ where $s,s^*$ are the respective nonzero elements of $x_3$ as defined in \eqref{eq:DFE3} which we recall here: \begin{equation}  \label{eq:DFS}\begin{split}
    s &= \tfrac 1 2\left( \tfrac{b}{\delta} + \tfrac{\delta + \nu}{\mu} - \sqrt{\left(\tfrac{b}{\delta} - \tfrac{\delta + \nu}{\mu} \right)^2 + \tfrac{4\xi b}{\mu} } \right), \\
    s^*&= \tfrac 1 2\left( \tfrac{b}{\delta} - \tfrac{\delta + \nu}{\mu} + \sqrt{\left(\tfrac{b}{\delta} - \tfrac{\delta + \nu}{\mu} \right)^2 + \tfrac{4\xi b}{\mu} } \right). \end{split}
\end{equation} We prove our result for any $\xi \in [0,1]$. As mentioned in proposition \ref{prop:DFEs}, if $\xi \to 0$, the duo $(s,s^*)$ alternately converges to $(\tfrac b \delta,0)$ if $\tfrac b \delta \le \tfrac{\delta+\nu}{\mu}$ or $(\tfrac{\delta + \nu}\mu,\tfrac b \delta - \tfrac{\delta+\nu}{\mu})$ if $\tfrac b \delta > \tfrac{\delta+\nu}{\mu}$. Because we are interested in the case of mixed disease free states (where some of the population is compliant and some is noncompliant), we assume that the latter condition holds.  We emphasize again that, while $X_{s,s^*}$ is an equilibrium point for the deterministic system \eqref{eq:SIRwithCompliance}, it is \emph{not} an equilibrium point for the stochastic system \eqref{eq:StochSIR} due to the stochastic drift term. However, since the deterministic dynamics are steady at $X_{s,s^*}$, it makes sense to try to quantify just how far the solution can stray from $X_{s,s^*}$ in terms of the the stochastic parameters $\sigma_\beta, \sigma_\mu$. This is the gist of the ensuing theorem. Before stating the theorem, we recall a version of the strong law of large numbers for continuous martingales, which will be useful in the proof.

\begin{lemma} \cite[Ch. 1, Thm. 3.4]{MaoBook} \label{lem:SLLN}
Suppose that $M = M(t)$ is a scalar, real-valued, continuous martingale vanishing at $t = 0$. Then $$\limsup_{t\to\infty} \frac{M(t)^2}{t} < \infty \,\,\, \text{a.s.} \,\,\,\,\,\, \implies \,\,\,\,\,\, \limsup_{t\to\infty} \frac{M(t)}{t} = 0 \,\,\, \text{a.s.}$$ 
\end{lemma} 

In particular, by theorem \ref{t.w09271}, any solution $(S,I,R,S^*,I^*,R^*)$ of \eqref{eq:StochSIR} remains bounded and nonnegative (a.s.) and thus for any continuous function $f(S,I,R,S^*,I^*,R^*)$, $$M(t) = \int^t_0 f(S,I,R,S^*,I^*,R^*) dW$$ is a continuous martingale which vanishes at $t = 0$ and satisfies $$\frac{M(t)^2}t = \frac{\int^t_0 f^2 dt}{t} \le \frac{\int^t C \, dt}{t} = C$$ where $C$ is an upper bound on $f^2$ for $0\le S,I,R,S^*,I^*,R^* \le \tfrac b \delta$. Thus $M(t)$ satisfies the hypotheses of lemma \ref{lem:SLLN} and $\limsup_{t\to\infty} (M(t)/t) = 0$ almost surely.  Because of this, in the work below, we will not be especially careful when we collect stochastic terms, since at the end, we will integrate, divide by $t$ and take a limit, whereupon these terms vanish. With this note, we state and prove our final theorem. 

\begin{theorem} \label{thm:Xssbehavior}
     Define $s,s^*$ as in \eqref{eq:DFS} and $\mathscr R_0(s,s^*)$ as in \eqref{eq:R0deterministic}, and assume that $\tfrac b \delta > \tfrac {\delta+\nu}\mu$. If $\mathscr R_0(s,s^*)<1$, then there is a positive constant $C$ (depending on the ambient parameters but independent of $\sigma_\beta,\sigma_\mu$) such that any solution $(S,I,R,S^*,I^*,R^*)$ of \eqref{eq:StochSIR} satisfies \begin{equation}
        \limsup_{t\to\infty} \frac 1 t \left(\int^t_0 (S-s)^2+I^2+R^2 + (S^*-s^*)^2+(I^*)^2+(R^*)^2)dt\right) \le C (\sigma_\beta^2 + \sigma_\mu^2)
     \end{equation}
\end{theorem}

\begin{remark}
This theorem encapsulates a stochastic version of theorem \ref{eq:stabilityOfDFE}(ii), and of theorem \ref{eq:stabilityOfDFE}(iii) in the general case (i.e., any $\xi \in (0,1]$ and $\nu \ge 0)$. Sending $\sigma_\beta, \sigma_\mu \to 0$ one recovers versions of theorems \ref{eq:stabilityOfDFE}(ii),(iii) which are stronger in the sense that they are global results, but weaker in the mode of convergence to the equilibrium point. We remark on the value of the constant $C$ after the proof. 
\end{remark}

\begin{proof}
The proof is quite intricate, but essentially proceeds by constructing another function which acts somewhat like a Lyapunov function, though the argument concludes differently. Rather than specify the  function at the beginning, we will construct it in several steps.  

We consider the new variables $u = S-s, v = I, w = R, u^* =  S^*-s^*, v^* = I^*, w^* = R^*.$ We also define $n = u+v+w$ and $n^* = u^*+v^*+w^*.$ We will again frequently use that $S,I,R,S^*,I^*,R^*$ remain nonnegative and that the total population is bounded by $\tfrac b \delta$. One difficulty here is that $u$ and $u^*$ do not have a definite sign (and this transfers to $n, n^*$) which makes it unclear how to bound certain terms. However, we do know that $u+s = S \ge 0$ and $u^*+s^* = S^* \ge 0$, which also gives $n + s \ge 0$ and $n^*+s^* \ge 0$.  

Taking derivatives and recalling that $s,s^*$ are defined so as to satisfy $(1-\xi)b-\mu ss^* + \nu s^* - \delta s = 0$ and $\xi b +\mu ss^* - (\nu+\delta)s^*=0$, we find that \begin{equation}
    \label{eq:sys47}\begin{split}
        du&= (-\beta(1-\alpha)(u+s)((1-\alpha)v+v^*) - \mu u(n^*+s^*) - \mu s n^* + \nu u^* - \delta u)dt \\
            &\hspace{2cm}+(-\sigma_\beta (1-\alpha)^2(u+s)v - \sigma_\mu (u+s)(u^*+s^*))dW\\
        dv&= (\beta(1-\alpha)(u+s)((1-\alpha)v+v^*)-\gamma v - \mu v(n^*+s^*) + \nu v^* - \delta v)dt \\
            &\hspace{2cm}+(\sigma_\beta (1-\alpha)^2(u+s)v - \sigma_\mu vv^*)dW\\
        dw&= (\gamma v - \mu w (n^*+s^*) +\nu w^* - \delta w)dt - \sigma_\mu ss^*dW \\
        du^*&=(-\beta (u^*+s^*)((1-\alpha)v+v^*) +\mu(u+s)(n^*+s^*) -(\nu+\delta)(u^*+s^*))dt \\
            &\hspace{2cm}+(-\sigma_\beta(u^*+s^*)v^* + \sigma_\mu (u+s)(u^*+s^*))dW\\
        dv^*&=(\beta (u^*+s^*)((1-\alpha)v+v^*) - \gamma v^* +\mu v (n^*+s^*) - (\nu+\delta)v^* \\
            &\hspace{2cm}+(\sigma_\beta (u^*+s^*)v^* + \sigma_\mu vv^*)dW\\
        dw^*&= (\gamma v^* + \mu w (n^*+s^*) - (\nu+\delta )w^*)dt + \sigma_\mu ww^* dW
    \end{split}
\end{equation} We first define $V_1 = (n+n^*)^2$, and we find \begin{equation} \label{eq:dV1}dV_1 = -2\delta(n+n^*)^2dt = (-2\delta n^2 - 2\delta (n^*)^2 - 4\delta nn^*)dt. \end{equation} 
Next, define $$V_2 = S^*+I^*+R^*-s^*-s^*\log(\tfrac{S^*+I^*+R^*}{s^*}) = n^* - s^*\log(\frac{n^*+s^*}{s^*}).$$ We note that $V_2 \ge 0$ for $n^* \ge -s^*$ (indeed, if we treat $V_2$ as a function of $n^*$, we find $V_2''\ge 0$ and $V_2' = V_2 = 0$ at $n^* = 0$, implying that $V_2$ takes a minimum value of $0$ at $n^* = 0$).  Also \begin{equation}
    \label{eq:dV2pre1} \begin{split}
        dV_2&= \bigg[\Big(1-\tfrac{s^*}{n^*+s^*}\Big)\Big(\mu n(n^*+s^*) +(\mu s-(\nu+\delta))n^* \Big) + \tfrac{(s^*)^2\sigma_\mu^2((u+s)(u^*+s^*)+vv^*+ww^*)^2}{2(n^*+s^*)^2} \bigg]dt\\
        &\hspace{2cm}\sigma_\mu\Big(1-\tfrac{s^*}{n^*+s^*}\Big)((u+s)(u^*+s^*)+vv^*+ww^*)) dW.
    \end{split}
\end{equation} Using $(u+s)(u^*+s^*)+vv^*+ww^*\le ((u+s)+v+w)((u^*+s^*)+v^*+w^*) = (n+s)(n+s^*)$ and $1-\tfrac{s^*}{n^*+s^*} = \frac{n^*}{n^*+s^*}$, we see that \begin{equation}\label{eq:dV2pre2}
\begin{split}
    dV_2&\le \left(\mu nn^* + \tfrac{s^*(\mu s - (\nu+\delta))}{(n^*+s^*)}(n^*)^2 + \tfrac{\sigma_\mu^2}{2}\Big(\tfrac b \delta \Big)^4\right)dt+ f(u,v,w,u^*,v^*,w^*)dW
\end{split}
\end{equation} where $f$ is a continuous function such that $\abs{f}$ is bounded by a polynomial in its variables (in particular, since $0\le S,I,R,S^*,I^*,R^*\le \tfrac b \delta$, we have $\abs{f} \le C$ for any of these values). In what follows, we will collect all stochastic terms in this manner: the function $f$ will change from line to line, but will be continuous (usually polynomial) and thus bounded on the solution values. In the end, these terms will disappear by lemma \ref{lem:SLLN}. 

Looking at \eqref{eq:dV2pre2}, we note that $\mu s - (\nu+\delta) \le 0$; this follows because, treating $s$ from \eqref{eq:DFS} a function of $\xi \in [0,1]$, it is decreasing in $\xi$, and thus reaches its maximum value of $s = \frac{\delta + \nu}{\mu}$ at $\xi = 0$. Hence we arrive at \begin{equation}
\label{eq:dV2}
dV_2 \le \left(\mu nn^* + \tfrac{\sigma_\mu^2}{2}\Big(\tfrac b \delta\Big)^4\right) dt + fdW.
\end{equation} Looking at \eqref{eq:dV1} and \eqref{eq:dV2}, we now have cross terms $nn^*$ with opposite signs, so we can balance these by multiplying by an appropriate constant, and arrive at  \begin{equation} \label{eq:dV12}d\Big(V_1+\tfrac{4\delta}\mu V_2\Big) \le \left(-2\delta n^2 -2\delta (n^*)^2 + \tfrac{2\delta \sigma_\mu^2}{\mu} \Big(\tfrac b \delta \Big)^4\right)dt + fdW  \end{equation} This gives us some control on $(n^*)^2$ which is necessary to control $u$, since $n^*$ appears in the equation for $u$. 

Now we define $V_3 = (u+v+u^*+v^*)^2$. Then \begin{equation} \label{eq:dV3pre} \begin{split}
dV_3 &= (-2\delta(u+v+u^*+v^*)^2-2\gamma(v+v^*)(u+v+u^*+v^*))dt \\ 
&\le \Big(-2\delta (u^2 + v^2+ (u^*)^2+ (v^*)^2) - 2(\gamma+2\delta)(uv+uv^* + u^*v+u^*v^*) -4\delta uu^*\Big)dt.
\end{split}
\end{equation} Here we have discarded $-4\delta vv^*$ which is the only cross term known to be nonpositive. We use the inequality $-2(\gamma+2\delta)uv \le \tfrac \delta 2 u^2 + \tfrac{2(\gamma+2\delta)^2}{\delta}v^2$ and similarly for the other cross terms involving the constant $(\gamma+2\delta).$ This results in \begin{equation}\label{eq:dV3}dV_3 \le \bigg(-\delta u^2 - \delta (u^*)^2 - \Big(2\delta - \tfrac{4(\gamma+2\delta)^2}{\delta}\Big)v^2 + \Big(2\delta - \tfrac{4(\gamma+2\delta)^2}{\delta}\Big)(v^*)^2 - 4\delta uu^*\bigg)dt. \end{equation} 

Next consider $V_4 = u^2$ so that \begin{align*}
    dV_4 &= \bigg[-2\beta(1-\alpha)u(u+s)((1-\alpha)v+v^*) - 2\mu u^2(n^*+s^*) - 2\mu s un^* + 2\nu uu^* - 2\delta u^2 \\ 
    &\hspace{2cm}+(-\sigma_\beta (1-\alpha)^2(u+s)v - \sigma_\mu (u+s)(u^*+s^*))^2\bigg]dt + fdW \\
    &\le \bigg[-2\beta(1-\alpha)^2uv(u+s) -2\beta(1-\alpha)uv^*(u+s) - 2\mu s un^* + 2\nu uu^* - 2\delta u^2 \\ 
    &\hspace{2cm}+(-\sigma_\beta (1-\alpha)^2(u+s)v - \sigma_\mu (u+s)(u^*+s^*))^2\bigg]dt + fdW
\end{align*} In the triple terms as the start, we use $u+s \le \tfrac b \delta$, $(1-\alpha) \le 1$,  as well as $$\abs{2\beta uv} \le \tfrac{\delta}{2}u^2 + \tfrac{2\beta^2}{\delta}v^2$$ and similarly for the $uv^*$ term. We also use $2\mu sun^* \le \delta u^2 + \tfrac{\mu^2 s^2}\delta (n^*)^2$, and bound the stochastic term using the total population bound. This yields \begin{equation} \label{eq:dV4} 
dV_4 \le \Big(\tfrac{2\beta^2}{\delta}v^2 + \tfrac{2\beta^2}{\delta}(v^*)^2 + \tfrac{\mu^2 s^2}{\delta}(n^*)^2 +2\nu uu^* + 2\Big( \tfrac b \delta\Big)^4(\sigma_\beta^2 + \sigma_\mu^2)\Big)dt + fdW
\end{equation} From \eqref{eq:dV3} and \eqref{eq:dV4}, we arrive at \begin{equation}
    \label{eq:dV34} \begin{split} 
    d\Big(V_3 + \tfrac {2\delta}\nu V_4\Big) &\le \bigg[-\delta u^2 -\delta(u^*)^2 + \tfrac{2 \mu^2 s^2}{\nu}(n^*)^2 + \tfrac {4\delta}\nu\Big( \tfrac b \delta\Big)^4(\sigma_\beta^2 + \sigma_\mu^2) \\ &\hspace{0.5cm}-  \Big(2\delta - \tfrac{4(\gamma+2\delta)^2+2\beta^2(2\delta/\nu)}{\delta}\Big)v^2 - \Big(2\delta - \tfrac{4(\gamma+2\delta)^2 +2\beta^2(2\delta/\nu)}{\delta}\Big)(v^*)^2\bigg]dt + fdW.
    \end{split}
\end{equation} 

Moving along we set $V_5 = (w+w^*)^2$ and find that \begin{equation} \label{eq:dV5}
\begin{split}
    dV_5 &= \Big(2\gamma(v+v^*)(w+w^*) - 2\delta(w+w^*)^2\Big)dt \\
    &\le \tfrac{4\gamma^2}{\delta} v^2 + \tfrac{4\gamma^2}{\delta} (v^*)^2 - \delta w^2 - \delta (w^*)^2.
\end{split}
\end{equation}

We now multiply \eqref{eq:dV12} by $\tfrac{\mu^2s^2}{\delta \nu}$ (so as to cancel the positive $(n^*)^2$ in \eqref{eq:dV34}), and add the result to \eqref{eq:dV34} and \eqref{eq:dV5}, to arrive at \begin{equation} \label{eq:dV12345} \begin{split}
d\Big(\tfrac{\mu^2s^2}{\delta \nu}\Big(V_1+\tfrac{4\delta}\mu V_2\Big) + V_3 + \tfrac {2\delta}\nu V_4 + V_5 \Big) &\le \bigg[-\delta u^2 - \delta (u^*)^2 - \delta w^2 - \delta (w^*)^2 \\ 
&\hspace{1cm}-  \Big(2\delta - \tfrac{4\gamma^2 + 4(\gamma+2\delta)^2+2\beta^2(2\delta/\nu)}{\delta}\Big)v^2\\
&\hspace{1cm}-  \Big(2\delta -\tfrac{4\gamma^2 + 4(\gamma+2\delta)^2+2\beta^2(2\delta/\nu)}{\delta}\Big)(v^*)^2\\
&\hspace{1cm}+ \tfrac{2\mu s^2+4\delta}{\nu}\Big(\tfrac b \delta\Big)^4(\sigma_\beta^2+\sigma_\mu^2)\bigg]dt + fdW.
\end{split}
\end{equation}  The last thing to do is to deal with the $v$ and $v^*$ equations. Using the linearization \eqref{eq:lin}, we see that \begin{equation}
    d\begin{bmatrix} v \\ v^*\end{bmatrix} = \bigg((F_{s,s^*}-V_{s,s^*}) \begin{bmatrix} v \\ v^*\end{bmatrix} + \begin{bmatrix} \beta(1-\alpha)u((1-\alpha)v+v^*) - \mu vn^* \\ \beta u^*((1-\alpha)v+v^*) + \mu vn^*\end{bmatrix} \bigg)dt + \begin{bmatrix} \Sigma_v \\ \Sigma_{v^*} \end{bmatrix} dW,
\end{equation} where $\Sigma_v= \sigma_\beta(1-\alpha)^2(u+s)v -\sigma_\mu vv^*$, $\Sigma_{v^*} = \sigma_\beta (u^*+s^*)v^* + \sigma_\mu vv^*.$ In particular, we have $\Sigma_v^2, \Sigma_{v^*}^2 \le 2\Big(\tfrac{b}{\delta}\Big)^4(\sigma_\beta^2+\sigma_\mu^2)$. Since we are assuming $\mathscr R_0(s,s^*) < 1$, we know that all eigenvalues of $M \defeq F_{s,s^*}-V_{s,s^*}$ have negative real part \cite{NextGenMat3}. Borrowing from the theory of Sylvester equations \cite[Ch. 5]{SE1},\cite[Ch. 8]{SE2}, this is a necessary and sufficient condition for the existence of a symmetric positive definite matrix $Q$ such that $M^tQ+QM=-I$. We define \begin{equation}
    \label{eq:V6def} V_6 = \innerprod{Q\begin{bmatrix} v\\v^*\end{bmatrix}}{\begin{bmatrix} v\\v^*\end{bmatrix}} = \|(v,v^*)\|^2_Q. 
\end{equation} A quick computation shows that \begin{equation}
    \label{eq:V6prebound} \begin{split}
    dV_6 &= \bigg(-v^2 - (v^*)^2 + 2\innerprod{Q\begin{bmatrix} v\\ v^*\end{bmatrix}}{ \begin{bmatrix} \beta(1-\alpha)u((1-\alpha)v+v^*) - \mu vn^* \\ \beta u^*((1-\alpha)v+v^*) + \mu vn^*\end{bmatrix}} \\ &\hspace{4cm} +  2 \innerprod{Q\begin{bmatrix} \Sigma_v \\ \Sigma_{v^*} \end{bmatrix}}{\begin{bmatrix} \Sigma_v \\ \Sigma_{v^*} \end{bmatrix}}\bigg)dt + f dW \\ 
    &= \bigg(-v^2 - (v^*)^2 + 2\innerprod{Q\begin{bmatrix} v\\ v^*\end{bmatrix}}{ \begin{bmatrix} \beta(1-\alpha)u((1-\alpha)v+v^*) - \mu vn^* \\ \beta u^*((1-\alpha)v+v^*) + \mu vn^*\end{bmatrix}} \\ &\hspace{4cm} +  8 \|Q\|_2\big(\tfrac{b}{\delta})^4(\sigma_\beta^2+\sigma_\mu^2)\bigg)dt + f dW \end{split}. 
\end{equation} For the remaining inner product term in \eqref{eq:V6prebound}, we use $v,v^* \le \tfrac b \delta$, as well as \begin{equation}
 (1-\alpha)u((1-\alpha)v+v^*) \le n(v+v^*), \,\,\,\,\,\,\,\, u^*((1-\alpha)v+v^*) \le  n(v+v^*).
\end{equation} Having done so, all remaining terms are bounded by constant multiples of $vn,v^*n,vn^*$ or $v^*n^*$, and we bound each of these using $\beta vn \le \tfrac{3\beta^2 n^2}{2\eps } + \tfrac {\eps v^2}6$, and similarly for other terms. All of this leads to the bound 
\begin{equation} \label{eq:dV6} \begin{split}
    dV_6 &\le \bigg(-v^2-(v^*)^2 + 2\|Q\|_2 \Big(\tfrac b \delta\Big) \bigg( \tfrac {2\eps} 3 v^2 + \tfrac \eps 3 (v^*)^2 + \tfrac{3\beta}{\eps} n^2 + \tfrac{3(\beta+\mu)}{\eps}(n^*)^2\bigg)\\ &\hspace{4cm} +  8 \|Q\|_2\Big(\tfrac{b}{\delta}\Big)^4(\sigma_\beta^2+\sigma_\mu^2)\bigg)dt + f dW   \\ 
    &\le \bigg(-\tfrac 1 3 v^2 - \tfrac 1 3 (v^*)^2 + 12\|Q\|^2_2\Big( \tfrac b \delta \Big)^2(\beta+\mu)(n^2 + (n^*)^2) dt \\ &\hspace{4cm}+ 8\|Q\|_2 \Big( \tfrac{b}\delta\Big)^4(\sigma_\beta^2+\sigma_\mu^2) \bigg)dt+ fdW,  \end{split}
\end{equation} where we chose $\eps = 1/(2\|Q\|_2 (b/\delta))$ in the last line.

We conclude by combining the bounds \eqref{eq:dV12}, \eqref{eq:dV12345} and \eqref{eq:dV6}. Specifically, define $$V = \Big(\tfrac{\mu^2s^2}{\delta \nu}\Big(V_1+\tfrac{4\delta}\mu V_2\Big) + V_3 + \tfrac {2\delta}\nu V_4 + V_5\Big) + c_1 \Big(V_1 + \tfrac {4\delta}\mu V_2\Big)+c_2 V_6.$$ This is nonnegative as long as $c_1,c_2$ are nonnegative. From the right hand sides of \eqref{eq:dV12}, \eqref{eq:dV12345} and \eqref{eq:dV6}, we see \begin{equation} \label{eq:dVFinalBoundPre}
    \begin{split}
        dV&\le \bigg(-\delta(u^2+(u^*)^2+w^2+(w^*)^2) +\Big(12\|Q\|^2_2\Big( \tfrac b \delta \Big)^2(\beta+\mu) c_2 - 2c_1\delta\Big)(n^2 + (n^*)^2) \\ 
        &\hspace{1cm} -\Big(2\delta - \tfrac{4\gamma^2 + 4(\gamma+2\delta)^2+2\beta^2(2\delta/\nu)}{\delta} +\frac{c_2}3\Big)(v^2+(v^*)^2) \\ &\hspace{1cm}+C(\sigma_\beta^2+\sigma_\mu^2)\bigg)dt + fdW\end{split}
\end{equation} Choosing $c_2 = \tfrac{12\gamma^2 + 12(\gamma+2\delta)^2+6\beta^2(2\delta/\nu)}{\delta}$ and then $c_1 = \frac{12\|Q\|^2_2(b/\delta)^2(\beta+\mu) c_2}{2\delta}$, we arrive at  \begin{equation} \label{eq:dVFinalBound}
    dV \le \Big(-\delta(u^2+v^2+w^2+(u^*)^2 + (v^*)^2+(w^*)^2)+ C(\sigma_\beta^2 + \sigma_\mu^2)\Big)dt + fdW.
    \end{equation} Integrating on $[0,t]$ and dividing by $t$ and $\delta$ shows that $$\frac{V(t)-V(0)}{t\delta } + \frac 1 t\int^t_0 \Big(u^2+v^2+w^2+(u^*)^2 + (v^*)^2+(w^*)^2\Big) \le C(\sigma_\beta^2+\sigma_\mu^2) + \frac{1}{t}\int^t_0 fdW.$$ Taking the $\limsup$ as  $t\to\infty$, the result follows since $V$ is nonnegative and the integral on the right hand side goes to zero almost surely by lemma \ref{lem:SLLN}.
\end{proof}

In particular, the constant $C$ in theorem \ref{thm:Xssbehavior} can be almost explicitly identified. The value that follows from the specific choices made in our proof is \begin{equation}\label{eq:CCC}C = \left(\tfrac{2\mu s^2+4\delta}{\nu}+ 48\|Q\|_2\tfrac{2\gamma^2 + 2(\gamma+2\delta)^2+\beta^2(2\delta/\nu)}{\delta}  \bigg(1 + \tfrac{4\|Q\|_2 (b/\delta)^2(\beta+\mu)}{3\mu}\bigg) \right) \left(\tfrac b \delta\right)^4\end{equation} which has not been optimized in any way, but still demonstrates how the constant should behave under different parameter regimes. The only non-explicit part of the constant is $\|Q\|_2$ where $Q$ is the solution of the Sylvester equation $M^tQ+QM=-I$ where $M \defeq F_{s,s^*}-V_{s,s^*}$. This matrix $Q$ is quaranteed to exist if and only if all eigenvalues of $F_{s,s^*} - V_{s,s^*}$ have negative real part which holds if and only if $\mathscr R_0(s,s^*)<1$. Thus we expect that this terms causes $C$ to blow up when $\mathscr R(s,s^*) \to 1$. Other things to notice is that $C\to \infty$ if $\nu \to 0, \mu \to 0$ or $\delta \to 0$ which makes intuitive sense.  If $\nu \to 0$ there is no recovery from noncompliance so everyone becomes noncompliant and we can no longer guarantee that $(S-s)$ remains small; if $\mu \to 0$, then no one is becoming noncompliant so this population dies out and we cannot guarantee that $(S^*-s^*)$ remains small; and as $\delta \to 0$, the upper bound on the population gets very large and thus the variance in the stochastic system also increases. 

\section{Simulations and discussion} \label{sec:sim}

We present some simulations to demonstrate various aspects of our model, and empirically verify the theorems. All simulations were performed in MATLAB. System \eqref{eq:SIRwithCompliance} was discretized using a simple explicit Euler scheme, and \eqref{eq:StochSIR} was discretized using Milstein's scheme \cite[Ch. 10]{KloedenBook}. 

For all simulations we consider the time interval $[0,50]$ with discretization parameter $\Delta t = 0.05$. For the first four figures, we have birth and death rate $b = \delta = 0.2$, disease recovery rate $\gamma = 1$, and reduction in infectivity $\alpha = 0.25$. For initial conditions, we use $S_0 = S^*_0 = 0.25, I_0 = I^*_0 = 0.25 - 10^{-8}, R_0 = R_0^* = 10^{-8}$. The small positive values of $R_0,R_0^*$ are used simply so that all initial conditions are positive and theorem \ref{t.w09271} holds. In this manner, at the outset, roughly half the population is infected and half the population is noncompliant. All of these choices are synthetic, and are not meant to reflect any specific real-world scenario, but rather to demonstrate the behavior of the model. The remaining parameters are the disease infectivity rate ($\beta$), the portion of newly introduced population which is noncompliant $(\xi)$, the infectivity and recovery rates for noncompliance ($\mu$ and $\nu$ respectively),  the levels of uncertainty in the infectivity rates for the disease ($\sigma_\beta$) and noncompliance ($\sigma_\mu$). These will be varied for different simulations (and in the final simulation, we will vary some of the preceding parameters as well).

In each figure, we include the solution of the deterministic system in dotted lines and one realization of the stochastic system in solid lines. We display the susceptible and infected classes, neglecting the recovered classes, simply because their behavior isn't particularly interesting and can be constructed from the others. 

In our first simulation, we take $\beta = 1$, $\xi = 0$, $\mu = \nu = 0.2$, and $\sigma_\beta = \sigma_\mu = 0.5$. Plugging in these values, we find that $1.625 = \tfrac b \delta + \frac{\sigma_\mu^2}{2\mu}\Big(\tfrac b \delta \Big)^2 < \frac{\nu + \delta}{\mu} = 2$ and $\mathscr R_0(\tfrac b \delta,0) < \mathscr R^\sigma_0(\tfrac b \delta) \approx 0.860.$ So the hypotheses of both theorem \ref{eq:stabilityOfDFE}(i) and theorem \ref{thm:E0expMSS} are met and the complaint disease free equilibrium $s = \tfrac b \delta, s^* = 0$ is stable. Indeed, this is born out in figure \ref{fig:1}, where $S^*,I,I^*$ all seemingly exhibit immediate exponential decay and $S$ simply increases to $s = \tfrac b \delta = 1$. 

\begin{figure}[b]
    \includegraphics[width=0.9\textwidth]{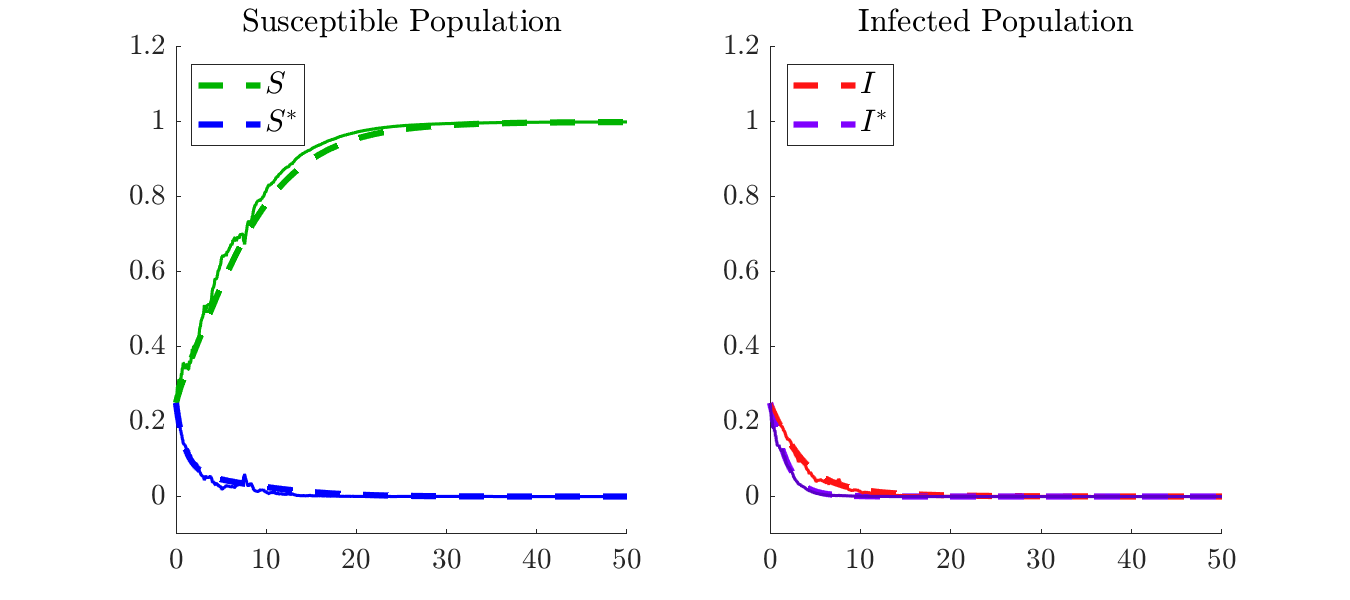}
    \caption{When the conditions of theorem \ref{thm:E0expMSS} (and thus theorem \ref{eq:stabilityOfDFE}(i)) are met, we see stability of the compliant disease free state $s = \tfrac b \delta, s^* = 0$ for both the deterministic and stochastic simulations.}
    \label{fig:1}
\end{figure}

For our second simulation, we take $\beta = 1, \xi = 0$, $\mu = \nu = 0.2$, and $\sigma_\beta = \sigma_\mu = 2$. Here, the deterministic part is the same as in the first simulation, but we have substantially increased the randomness. In this case, $1 = \tfrac b \delta < \tfrac{\nu+\delta}{\mu} = 2$ and $\mathscr R_0(\tfrac b \delta,0) \approx 0.803$ so the conditions of theorem \ref{eq:stabilityOfDFE}(i) are met, and stability is maintained in the deterministic case, but $11 = \tfrac b \delta + \frac{\sigma_\mu^2}{2\mu}\Big(\tfrac b \delta \Big)^2 > \frac{\nu+\delta}{\mu} = 2$ and $\mathscr R_0^\sigma(\tfrac b \delta,0) \approx 1.708$ so neither of the conditions for theorem \ref{thm:E0expMSS} are met and we cannot guarantee stability in the stochastic case. In this case, the stochastic simulations often exhibit nontrivial outbreaks of the disease as seen in figure \ref{fig:2}, where there is an initial outbreak of the disease among the noncompliant population, and later a significant outbreak among the compliant population that correlates with a random spike in noncompliance. Interestingly, in all realizations we observed, by time $T = 50$, the stochastic simulation seems to match the deterministic one, which gives empirical evidence that the sufficient conditions for stability in theorem \ref{thm:E0expMSS} may not be necessary (though it is not clear how they could be relaxed in our proof). 

\begin{figure}[t]
    \includegraphics[width=0.9\textwidth]{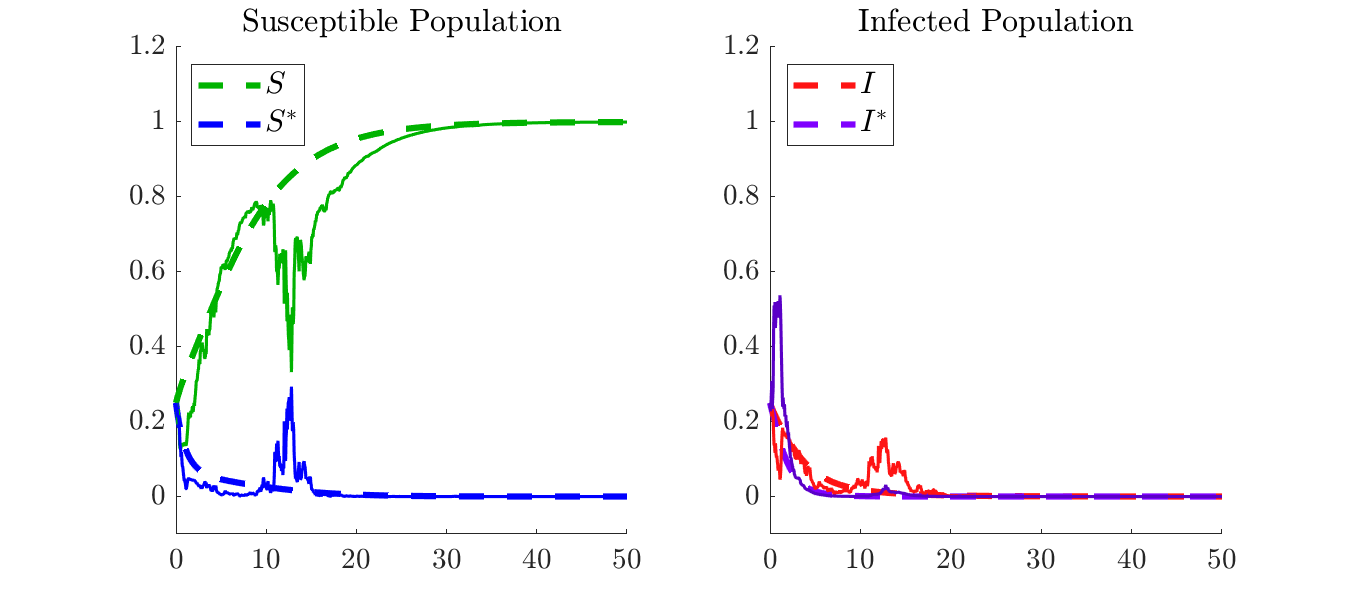}
    \caption{When the conditions of \ref{eq:stabilityOfDFE}(i) are met but the conditions of theorem \ref{thm:E0expMSS} are not met, we can only guarantee stability in the deterministic case. Because of this, we often observe significant outbreaks of the disease, as see in the right panel above. Here there is an immediate outbreak of the disease among the noncompliant population, and later an outbreak among the compliant population which is correlated with a random spike in noncompliance.}  
    \label{fig:2}
\end{figure}

For our third simulation, we consider $\beta = 0.6, \xi = 1$, $\mu = 0.1$, $\nu = 0$, and $\sigma_\beta = \sigma_\mu = 0.4$. In this case, $\mathscr R_0^\sigma(0,\tfrac b \delta) \approx 0.971$ and $0.16 = \Big(\tfrac \beta \delta\Big)^2\sigma_\mu^2 < \delta = 0.2$, so the conditions of theorem \ref{thm:WorstCase} are met (as well as those of theorem \ref{eq:stabilityOfDFE}(iii) in the worst case scenario of $\xi = 1$ and $\nu = 0$), so the disease free equilibrium $s = 0, s^* = \tfrac b \delta$ is stable for both the deterministic and stochastic systems. This is seen in figure \ref{fig:3}. If however, we increase $\sigma_\mu = 2$ (leaving all other parameters as specified), we still have $\mathscr R_0^\sigma(0,\tfrac b \delta) \approx 0.971$, but now $4 = \Big(\tfrac \beta \delta\Big)^2\sigma_\mu^2 > \delta = 0.2$. Thus according to theorem \ref{thm:WorstCase}, the infections should still exhibit exponential decay, but we can no longer guarantee stability for the DFE. This is seen in figure \ref{fig:4} where the infections do indeed die out, but there is some initial fighting between $S$ and $S^*$. As mentioned in remark \ref{rem:Thm44}, this would be unlikely to cause any concern for the policy-maker, since the important consideration is whether or not infections persist, not the particular  concentration of compliance and noncompliance once infections have died out, and any additional compliant population is actually a boon since they will reduce the effective reproductive ratio.

\begin{figure}[t]
    \includegraphics[width=0.9\textwidth]{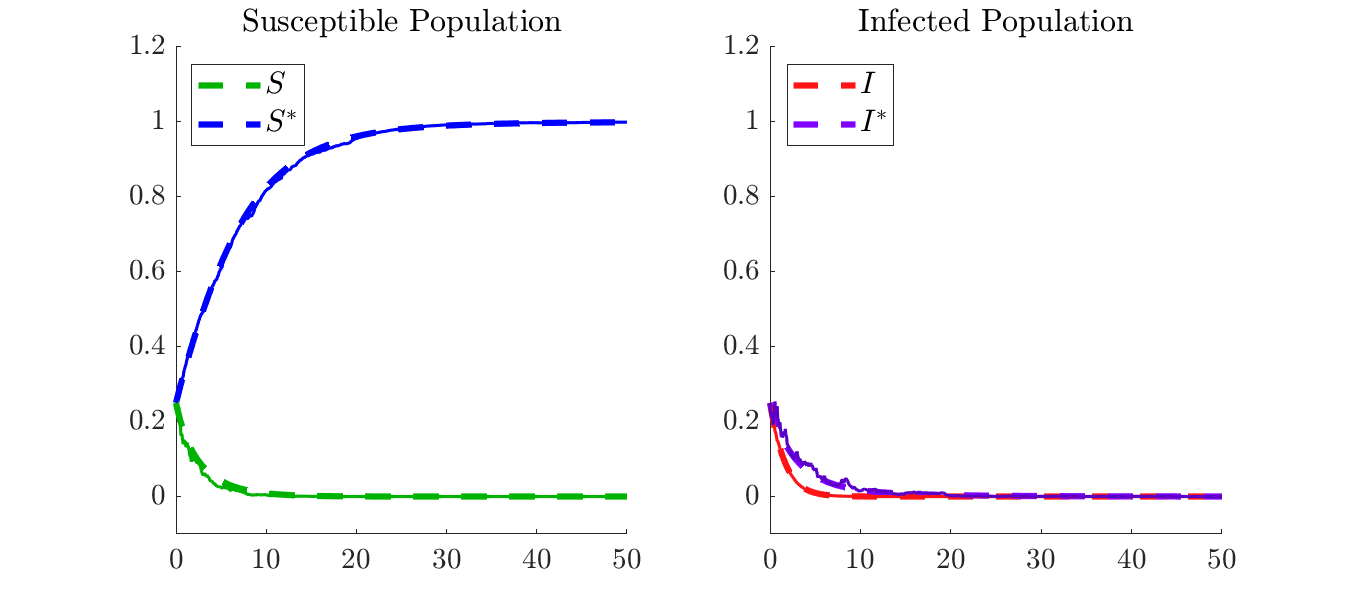}
    \caption{In the case of $\xi = 1$ and $\nu = 0$, when the conditions of theorem \ref{thm:WorstCase} are met, the disease free equilibrium $s = 0, s^* = \tfrac b \delta$ is stable for both the stochastic and deterministic systems.}  
    \label{fig:3}
\end{figure}

\begin{figure}[t]
    \includegraphics[width=0.9\textwidth]{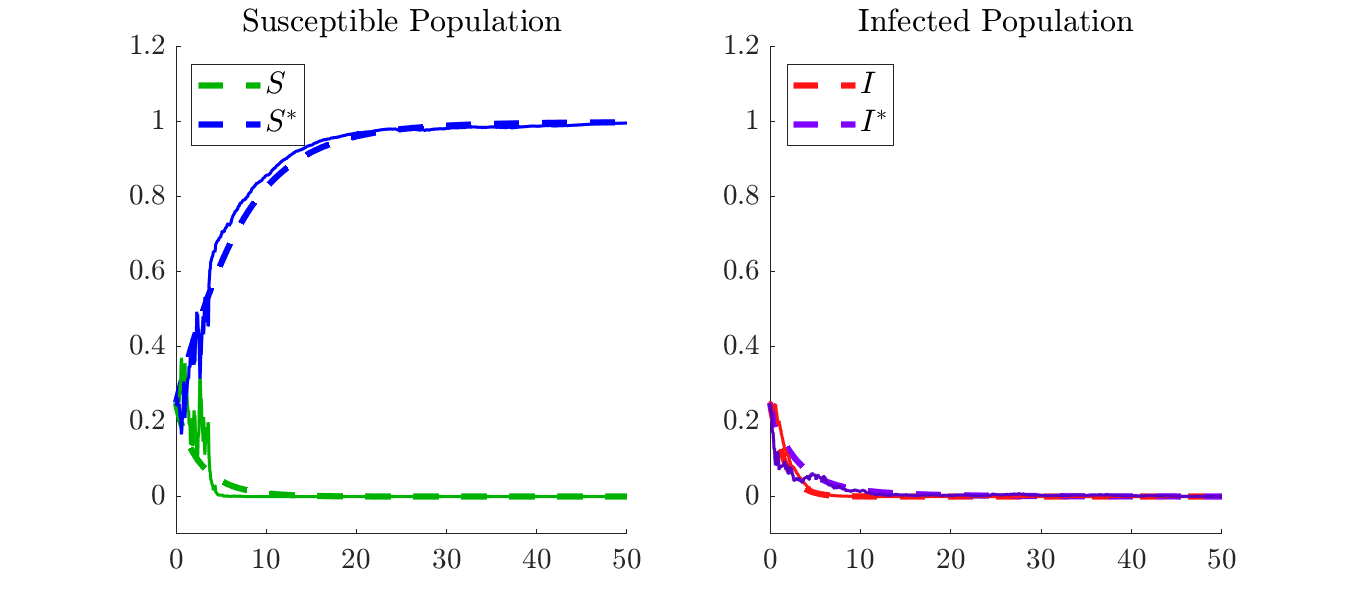}
    \caption{In the case of $\xi= 1$ and $\nu = 0$, when $\mathscr R_0^\sigma(0,\tfrac b \delta) <1$, but $\Big(\tfrac b \delta \Big)^2\sigma_\mu^2 > \delta$ the disease free equilibrium $s = 0, s^* = \tfrac b \delta$ is stable for the deterministic system. It may not be stable for the stochastic system, but regardless, by theorem \ref{thm:WorstCase}, we can still guarantee that infections die out. In this case, the particular dynamics of $S$ and $S^*$, which are somewhat unpredictable especially at the outset, would likely not concern the policy-maker.}  
    \label{fig:4}
\end{figure}

In our final simulation, we vary several of the parameters (including some that were fixed in figures \ref{fig:1}-\ref{fig:4}). Specifically, we set the birth rate to $b = 0.3$, the death rate to $\delta = 1$, the infectivity rate of the disease to $\beta = 0.5$, the recovery rate for the disease to $\gamma = 0.25$, the infectivity rate for noncompliance to $\mu = 8,$ the recovery rate from noncompliance to $\nu = 1$, and the uncertainty in the infectivity rates to $\sigma_\beta = \sigma_\mu = 0.125$. We take $\xi = 0$ so that newly introduced members of the population are uniformly compliant. The reduction in infectivity is held at $\alpha = 0.25$.  With all these decisions, we consider the disease free state $s = \frac{\nu+\delta}{\mu} = \frac{1}{4}$ and $s^* = \tfrac b \delta - \frac{\nu+\delta}{\mu} = \frac{1}{20}$. In this case $\mathscr R_0(s,s^*) \approx 0.0772$ so we are well within the local asymptotic stability regime provided by theorem \ref{eq:stabilityOfDFE}(ii). Thus the solution of the deterministic system should tend toward the disease free equilibrium. While the stochastic system will not be at equilibrium at this pair $(s,s^*)$, we do satisfy the hypothesis of theorem \ref{thm:Xssbehavior}, and taking $C$ from \eqref{eq:CCC}, we have $C(\sigma_\beta^2+\sigma_\mu^2) \approx 0.0171$. So while the solution of the stochastic system will not be steady, it should stay within these error bars of the deterministic solution. This is seen to occur in figure \ref{fig:5} where we plot the deterministic solution along with three realizations of the stochastic system for times $t \ge 10$. We see that the stochastic realizations stay well within the guaranteed error bars, which are displayed in gray. 

\begin{figure}[t]
    \includegraphics[width=0.9\textwidth]{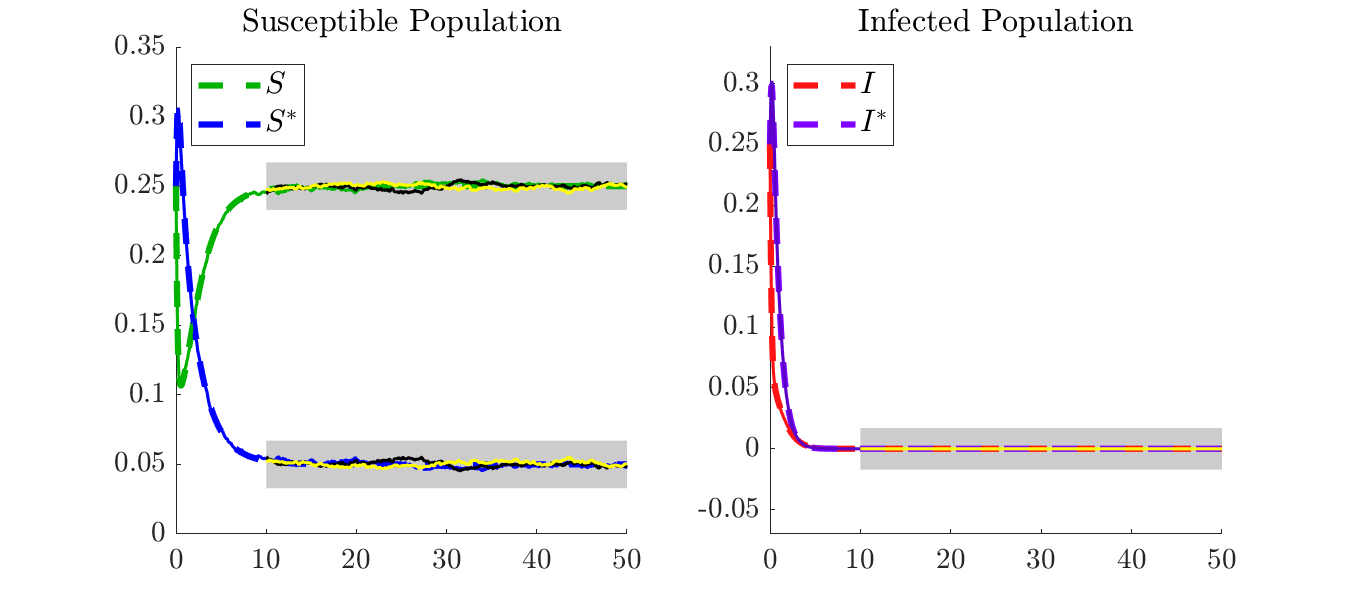}
    \caption{The mixed disease free state $(s,s^*)$ defined in \eqref{eq:DFS} is a equilibrium point for the deterministic model, but \emph{not} for the stochastic model. When $\mathscr R_0(s,s^*)<1$, theorem \eqref{eq:stabilityOfDFE} guarantees stability of this DFE for the deterministic system, and theorem \ref{thm:Xssbehavior} gives error bars for the solution to the stochastic system for large time (represented in gray). Plotted are the deterministic solution, and three realizations of the stochastic system (for times $t \ge 10$). Note the modified axes when compared with other figures.}  
    \label{fig:5}
\end{figure}

\section{Concluding remarks} \label{sec:conclusion}

In this manuscript we analyze ODE and stochastic ODE models for epidemiology incorporating human behavior in the form of noncompliance with NPIs. After analyzing DFEs for the deterministic model, we prove long time existence and positivity for the stochastic model and establish stability conditions for two DFEs by constructing suitable stochastic Lyapunov functions. Finally, we consider points which are DFEs of the deterministic model but not the stochastic model, and we quantify how much the stochastic model may stray from these points in the parameter regimes where they are stable for the deterministic model. We demonstrate our results with simulations. 

{\CPEDIT From a epidemiological standpoint, our results indicate that, comparing a deterministic model to a stochastic one wherein infection rates (both for the disease and for the social contagion of noncompliant behavior) are modified with Gaussian white noise, the extinction criteria for the disease are strictly stronger in the stochastic case. Mathematically, this is seen in that sufficient conditions for disease extinction in theorems \ref{thm:E0expMSS} and \ref{thm:WorstCase} are strictly stronger than their deterministic analogs. Likewise conditions which prevent the spread of noncompliance must be strengthened in the stochastic case. In some sense, this indicates that, while the white noise can either raise or lower infection rates, the policy-maker would need to plan for the worst.} 

We especially emphasize that theorems theorems \ref{thm:E0expMSS} and \ref{thm:WorstCase} provide \emph{sufficient} conditions for disease extinction. As such, they do not address disease persistence and may be able to relaxed. For example, in \cite{Gray,reducedR0} the authors derive thresholds for disease extinction or persistence wherein extinction becomes more likely in the presence of noise. However, the models in \cite{Gray,reducedR0} are somewhat simpler especially in that the dynamics of infected class have the form $dI = I\cdot(f(S,I)dt+g(S,I)dW)$, and this proportionality to $I$ is heavily exploited in the analysis. Because our model has competing nonlinearities and does not have this structure for the infected class dynamics, it is not straightforward to adapt the analysis \cite{Gray,reducedR0} to our case, and we instead have the stronger sufficient conditions for extinction as in \cite{JJS,vetro}.

This work could be extended in many ways. One would be the identification of and stability conditions for endemic equilibrium points for the deterministic system and quantification of the behavior of stochastic system near these points. Another would be the incorporation of spatial effects resulting in a system of stochastic partial differential equation. There are some deterministic results in this direction in \cite{BPW,PW} but adding stochasticity would lead to nontrivial complications in the analysis. Additionally, the inclusion of nonlocal effects wherein the disease and/or noncompliance can spread nonlocally could lead to interesting models requiring significantly modified analysis, and may be apt in the modeling of spread of attitudes, where social media, news coverage or a variety of other sources may be seen to influence behavior without direct contact. Another direction would be to explore the effects of introducing different types of stochastic perturbation, such as Stratonovich noise as in \cite{BL,CMZ,FWZ,LM}, modified infections rates via Ornstein-Uhlenbeck processes as in \cite{LP,LiuOU}, or a mix of continuous- and discrete-time perturbations as in \cite{alex}. One last path forward would be to consider the governmental protocols as control variables and design a cost functional for the policy-maker so that this becomes a stochastic control problem.

\section*{Acknowledgments}
The authors would like to thank three anonymous reviewers for their insightful comments and remarks. WW was partially supported by a Junior Faculty Fellowship from the University of Oklahoma and a Simons Foundation grant (No. 0007730).

	\bibliographystyle{siam}
	\bibliography{ref}

\bigskip\noindent 
Christian Parkinson, Department of Mathematics \& Department of Computational Mathematics, Science and Engineering, Michigan State University, East Lansing, MI,  USA;
e-mail: \url{chparkin@msu.edu}

\medskip\noindent 
Weinan Wang, Department of Mathematics, University of Oklahoma, Norman, OK, USA;
 e-mail: \url{ww@ou.edu}

\end{document}